\documentclass[11pt, a4paper]{article}
\usepackage[latin1]{inputenc}
\usepackage[T1]{fontenc} 
\usepackage{amsmath}
\usepackage{amsfonts}
\usepackage{amssymb}
\usepackage{amsthm}
\usepackage[affil-it]{authblk}
\usepackage{enumitem}
\usepackage[normalem]{ulem}
\usepackage[mathscr]{euscript}
\usepackage{dsfont}
\usepackage{float}
\usepackage{times}
\usepackage{bbm}
\usepackage{xcolor}
\usepackage{xparse}
\usepackage[hang]{footmisc}
\usepackage[colorlinks,linkcolor=blue,citecolor=blue,urlcolor=blue]{hyperref}
\usepackage[capitalise]{cleveref}
\usepackage[left=3cm, right=3cm, bottom=3cm, top=3cm]{geometry}

\usepackage[backend=biber, style=alphabetic, giveninits=true, maxbibnames=99, maxalphanames=99, url=false, isbn=false, doi=false, hyperref=false]{biblatex}
\newlist{sk}{enumerate}{1}
\setlist[sk]{label=$\mathbf{(S)}_{k,H}$:, ref=$\mathbf{(S)}_{k,H}$, wide, labelwidth=!, labelindent=0pt}

\newlist{cond_H}{enumerate}{1}
\setlist[cond_H]{label=$\mathbf{(H)}$:, ref=$\mathbf{(H)}$, wide, labelwidth=!, labelindent=0pt}

\newlist{psk_loc}{enumerate}{1}
\setlist[psk_loc]{label=$\mathbf{(PS_{\mathrm{loc}})}_{k,H}$:, ref=$\mathbf{(PS_{\mathrm{loc}})}_{k,H}$, wide, labelwidth=!, labelindent=0pt}

\newlist{psk}{enumerate}{1}
\setlist[psk]{label=$\mathbf{(PS)}_{k,H}$:, ref=$\mathbf{(PS)}_{k,H}$, wide, labelwidth=!, labelindent=0pt}

\newtheorem{theorem}{Theorem}[section]
\newtheorem{lemma}[theorem]{Lemma}
\newtheorem{proposition}[theorem]{Proposition}
\newtheorem{corollary}[theorem]{Corollary}

\theoremstyle{definition}
\newtheorem{definition}[theorem]{Definition}

\theoremstyle{remark}
\newtheorem{remark}[theorem]{Remark}

\numberwithin{equation}{section}

\crefname{example}{Example}{Examples}
\Crefname{example}{Example}{Examples}

\crefname{assumption}{Assumption}{Assumptions}
\Crefname{assumption}{Assumption}{Assumptions}

\crefname{condition}{Condition}{Conditions}
\Crefname{condition}{Condition}{Conditions}

\let\para\S
\crefformat{section}{#2\para#1#3}
\crefformat{subsection}{#2\para#1#3}
\crefformat{subsubsection}{#2\para#1#3}

\setlist{topsep=1ex, itemsep=0.5ex, before={\setlist{topsep=-.5ex}}}

\DeclareMathOperator{\id}{id}

\newcommand{\f}{\ensuremath{\frac}}
\newcommand{\ft}{\ensuremath{\mathfrak{t}}}

\newcommand{\cB}{\ensuremath{\mathscr{B}}}
\newcommand{\B}{\ensuremath{\mathcal{B}}}
\newcommand{\cC}{\ensuremath{\mathscr{C}}}
\newcommand{\C}{\ensuremath{\mathcal{C}}}

\newcommand{\Cloc}{\ensuremath{\mathcal{C}_{\mathrm{loc}}}}

\newcommand{\E}{\ensuremath{\mathcal{E}}}

\newcommand{\cN}{\ensuremath{\mathcal{N}}}

\newcommand{\fK}{\ensuremath{\mathfrak{K}}}
\newcommand{\fL}{\ensuremath{\mathfrak{L}}}
\renewcommand{\S}{\ensuremath{\mathcal{S}}}

\newcommand{\Q}{\ensuremath{\mathbb{Q}}}

\newcommand{\F}{\ensuremath{\mathcal{F}}}

\newcommand{\x}{\ensuremath{\mathrm{x}}}

\newcommand{\cI}{\ensuremath{\mathscr{I}}}
\newcommand{\I}{\ensuremath{\mathcal{I}}}

\newcommand{\balpha}{\ensuremath{\boldsymbol{\alpha}}}
\newcommand{\bbeta}{\ensuremath{\boldsymbol{\beta}}}

\newcommand{\bigcdot}{\boldsymbol{\cdot}}

\renewcommand{\L}{\ensuremath{\mathcal{L}}}

\newcommand{\N}{\ensuremath{\mathbb{N}}}
\newcommand{\cP}{\ensuremath{\mathscr{P}}}
\renewcommand{\P}{\ensuremath{\mathcal{P}}}
\newcommand{\R}{\ensuremath{\mathbb{R}}}
\newcommand{\X}{\ensuremath{\mathcal{X}}}
\newcommand{\1}{\ensuremath{\mathds{1}}}
\newcommand{\T}{\ensuremath{\mathcal{T}}}
\newcommand{\Y}{\ensuremath{\mathcal{Y}}}

\newcommand{\sW}{\ensuremath{\mathsf{W}}}

\def\<{\langle}
\def\>{\rangle}

\NewDocumentCommand{\Lin}{om}{\IfNoValueTF{#1}{L(\R^{#2},\R^{#2})}{L(\R^{#1},\R^{#2})}}
\NewDocumentCommand{\Cb}{om}{\IfNoValueTF{#1}{\C_b^{#2}}{\C_b^{#2,#1}}}

\def\le{\leq}
\NewDocumentCommand{\Lip}{om}{\IfNoValueTF{#1}{|#2|_{\mathrm{Lip}}}{|#2|_{\mathrm{Lip};\,#1}}}

\newcommand{\define}{\ensuremath\triangleq}
\newcommand{\expec}[1]{\mathbb{E}[#1]}
\newcommand{\Expec}[1]{\mathbb{E}\left[#1\right]}

\newcommand{\braket}[1]{\ensuremath\langle#1\rangle}
\newcommand{\Braket}[1]{\ensuremath\left\langle#1\right\rangle}
\newcommand{\prob}{\ensuremath\mathbb{P}}
\newcommand{\tv}[1]{\ensuremath{\|#1\|_{\mathrm{TV}}}}
\newcommand{\TV}[1]{\ensuremath{\left\|#1\right\|_\mathrm{TV}}}

\renewcommand{\geq}{\geqslant}
\renewcommand{\leq}{\leqslant}
\def\Cb{{\mathrm {BC}}}
\def\${|\!|\!|}
\def\B{{\mathcal B}}
\def\F{{\mathcal F}}
\def\ge{\geq}
\newcommand{\vertiii}[1]{{\left\vert\kern-0.25ex\left\vert\kern-0.25ex\left\vert #1 \right\vert\kern-0.25ex\right\vert\kern-0.25ex\right\vert}}
\newcommand{\rom}[1]{(\textup{\uppercase\expandafter{\romannumeral#1}})}

\makeatletter
\newcommand{\substackal}[1]{%
  \vcenter{%
    \Let@ \restore@math@cr \default@tag
    \baselineskip\fontdimen10 \scriptfont\tw@
    \advance\baselineskip\fontdimen12 \scriptfont\tw@
    \lineskip\thr@@\fontdimen8 \scriptfont\thr@@
    \lineskiplimit\lineskip
    \ialign{\hfil$\m@th\scriptstyle##$&$\m@th\scriptstyle{}##$\hfil\crcr
      #1\crcr
    }%
  }%
}
\makeatother

\newcommand\blfootnote[1]{%
  \begingroup
  \renewcommand\thefootnote{}\footnote{#1}%
  \addtocounter{footnote}{-1}%
  \endgroup
}

\begin{document}
\title{On the (Non-)Stationary Density of Fractional-Driven Stochastic Differential Equations}
\author[1,3]{Xue-Mei Li}
\author[2]{Fabien Panloup}
\author[3]{Julian Sieber}
\affil[1]{\'Ecole Polytechnique F\'ed\'erale de Lausanne, Switzerland}
\affil[2]{Univ Angers, CNRS, LAREMA, SFR MATHSTIC, F-49000 Angers, France}
\affil[3]{Department of Mathematics, Imperial College London, 180 Queen's Gate, London SW7 2RH, United Kingdom}
\date{\today}

\newgeometry{top=0cm, bottom=1.5cm}

\maketitle
\thispagestyle{empty}

\blfootnote{Email addresses: \href{mailto:fabien.panloup@univ-angers.fr}{fabien.panloup}@univ-angers.fr, \{\href{mailto:xue-mei.li@imperial.ac.uk}{xue-mei.li}, \href{mailto:j.sieber19@imperial.ac.uk}{j.sieber19}\}@imperial.ac.uk. \\[1ex] \scriptsize
  XML has been supported by the EPSRC grants EP/V026100/1 and EP/S023925/1. JS has been supported by G-Research and the EPSRC Centre for Doctoral Training in Mathematics of Random Systems: Analysis, Modelling and Simulation (EP/S023925/1).
}

\vspace{-3em}

\begin{abstract}
  \noindent
  We investigate the stationary measure $\pi$ of SDEs driven by additive fractional noise with any Hurst parameter and establish that $\pi$ admits a smooth Lebesgue density obeying both Gaussian-type lower and upper bounds. The proofs are based on a novel representation of the stationary density in terms of a Wiener-Liouville bridge, which proves to be of independent interest: We show that it also allows to obtain Gaussian bounds on the non-stationary density, which extend previously known results in the additive setting. In addition, we study a parameter-dependent version of the SDE and prove smoothness of the stationary density, jointly in the parameter and the spatial coordinate. With this we revisit the fractional averaging principle of Li and Sieber [Ann. Appl. Probab. 32(5) (2022)] and remove an ad-hoc assumption on the limiting coefficients. Avoiding any use of Malliavin calculus in our arguments, we can prove our results under minimal regularity requirements.\\[1ex]
\end{abstract}

{
\small
\hypersetup{hidelinks}
\setcounter{tocdepth}{2}
\tableofcontents
}

\newpage

\restoregeometry

\section{Introduction and Main Results}

The aim of the article is to investigate the stationary solution to the stochastic differential equation (SDE)
\begin{equation}\label{eq:sde}
  dY_t=b(Y_t)\,dt+\sigma\,dB_t 
\end{equation}
as well as the sensitivity of the stationary measure of its parameter-dependent version
\begin{equation}\label{eq:parameter_sde}
  dY_t^\lambda=b(\lambda, Y_t^\lambda)\,dt+\sigma\,d B_t, \qquad Y_0^\lambda=Y_0,\qquad \lambda\in\R^d.
\end{equation}
Here, $\sigma$ is a non-degenerate $n\times n$-matrix, $b:\R^n\to\R^n$, and $(B_t)_{t\geq 0}$ is an $n$-dimensional fractional Brownian motion (fBm) with Hurst parameter $H\in(0,1)$. As an application, we study a two-scale stochastic dynamics describing a non-equilibrium molecular evolution, see \cref{cor:averaging}. 
 
The fBm is a commonly viewed as the simplest stochastic process modelling time-correlated noise. Generalizing the standard Wiener process ($H=\f 12$), a one-dimensional fractional Brownian motion $(B_t)_{t\geq 0}$ with Hurst parameter $H\in (0,1)$ is a centered, self-similar Gaussian process with stationary increments: $\Expec{|B_t-B_s|^{2H}}=|t-s|^{2H}$. However, the increments are correlated with a power law correlation decay. This is the dominant feature of the equation \eqref{eq:sde}. Hence, we shall focus on the non-Markovian case, $H\neq\f 12$, in our subsequent analysis. 

Throughout this article we impose sufficient regularity conditions on $b:\R^n\to\R^n$ so that \eqref{eq:sde} is pathwise well posed and has a uique stationary solution. As shown in \cite{Hairer2005}, these hold if the drift is \emph{contracting off the diagonal}, see \eqref{eq:off_diagonal_contraction}. Slightly weaker conditions are given in \cite{Hairer2007,Deya2019}. 

 More formally, stationarity of \eqref{eq:sde} means that there is a measure $\prob_\pi\in\P\big(\C(\R_+,\R^n)\big)$ such that the following hold:
\begin{itemize}
  \item If $\bar{Y}_t:\C(\R_+,\R^n)\to\R^n$ is the coordinate process, $\bar{Y}_t(\omega)\define\omega(t)$, then the process
  \begin{equation*}
     \sigma^{-1}\Bigl(\bar{Y}_t-\bar{Y}_0-\int_0^t b(\bar{Y}_s)\,ds\Bigr)_{t\geq 0}
  \end{equation*}
  is an $H$-fBm under $\prob_\pi$.
  \item For $t\geq 0$ let $\T_t:\C(\R_+,\R^n)\to\C(\R_+,\R^n)$, $\T_t\omega\define\omega(\cdot+t)$ be the canonical time-shift. Then $\T_t^{*}\prob_\pi=\prob_\pi$ for each $t\geq 0$.
\end{itemize}
In other words, $\prob_\pi$ is the path space law of a strictly stationary solution to \eqref{eq:sde}. We denote its marginal by $\pi=\prob_{\pi}(\bar{Y}_0\in\cdot)\in\P(\R^n)$.

Under the aforementioned conditions $\pi$ admits a density $p_\infty$ with respect to the $n$-dimensional Lebesgue measure. Indeed, it can be shown that the solution to \eqref{eq:sde} started from $Y_0=0$ is non-degenerate in the sense of Malliavin calculus, which implies that $\L(Y_t)\ll\lambda^n$ at \emph{finite} times $t>0$. The existence of the density $p_\infty$ thus follows from the fact that $\TV{\L(Y_t)-\pi}\to 0$ as $t\to\infty$, see \cite{Hairer2005}.

\medskip

With regards to the stationary densities of \eqref{eq:sde} and \eqref{eq:parameter_sde}, respectively, we address the following questions:

\begin{enumerate}[label=($\mathrm{Q}_{\arabic*}$)]
  \item\label{it:q1} Under what conditions is the mapping $y\mapsto p_\infty(y)$ smooth? 
  \item\label{it:q2} Does the density $p_\infty$ have Gaussian tails, that is, are there $c_1,C_1>0$ such that $p_\infty(y)\leq C_1 e^{-c_1|y|^2}$ for each $y\in\R^n$?
  \item\label{it:q3} Is the density positive everywhere, that is, $p_\infty(y)>0$ for (almost) every $y\in\R^n$ and, if so, does it even satisfy a Gaussian-type lower bound, that is, are there $c_2,C_2>0$ such that $p_\infty(y)\geq C_2 e^{-c_2|y|^2}$ for each $y\in\R^n$?
  \item\label{it:q4}  
  Let us now suppose that the drift $b:\R^d\times\R^n\to\R^n$ is parameter-dependent and consider \eqref{eq:parameter_sde}. Assuming that, for each $\lambda\in\R^d$, there is a unique stationary path space law $\prob_{\pi^\lambda}$ with marginal density $p^\lambda_\infty$, what can be said about the regularity of the mapping
  \begin{equation*}
    (\lambda,y)\mapsto p_\infty^\lambda(y)?
  \end{equation*}
\end{enumerate} 

The first three questions have been addressed for \emph{non-stationary} solutions of \eqref{eq:sde} and even for the case of multiplicative noise:

\begin{enumerate}[label=($\mathrm{Q}_{\arabic*}$)]
  \item The existence of the Lebesgue density $\bar p_t(y_0;\cdot):\R^n\to\R_+$ of the solution to \eqref{eq:sde} started from a deterministic initial condition $y_0\in\R^n$  was established by Nualart and Saussereau \cite{Nualart2009} for elliptic (multiplicative) noise with Hurst parameters $H>\frac12$, see also \cite{Nourdin2006,Hu2007}.  The density $\bar p_t(y_0; \cdot)$ was shown to be smooth by Baudoin and Hairer \cite{Baudoin2007}  under  H\"ormander's condition on the vector fields. This was later extended to $H>\frac13$ in Hairer and Pillai \cite{Hairer2013}; it actually extends to $H>\frac14$ as shown in Cass, Hairer, Litterer, and Tindel \cite{Cass2015}. 

  \item We know only two pieces of works on  Gaussian-type upper bounds for $\bar p_t(y_0;\cdot)$:  The first one is  Baudoin, Ouyang, and Tindel \cite{Baudoin2014} where they proved an upper bound for $H>\frac13$ under a skew-symmetry assumption. This  bound was later generalized to the elliptic case in Baudoin, Nualart, Ouyang, and Tindel \cite{Baudoin2016}.
 
  \item  The strict positivity of the smooth version of $\bar p_t(y_0; \cdot)$ was shown in \cite{Baudoin2016} in the elliptic case for $H>\frac13$, see also the recent preprint \cite{Geng2020}. This was extended to the hypoelliptic case for $H>\frac14$ by Geng, Ouyang, and Tindel \cite{Geng2020hypo}. A Gaussian-type lower bound was proven by Besal\'u, Kohatsu-Higa, and Tindel \cite{Besalu2016} for both one-dimensional systems and for elliptic multiplicative noise for  $H>\frac12$.
\end{enumerate} 

Related problems for the stationary measures on the other hand have been much less studied. Starting with the seminal work of Hairer \cite{Hairer2005}, existence of an invariant measure and convergence rates were obtained in \cite{Panloup2020,Li2020} for additive noise under a contraction condition and in \cite{Fontbona2017,Deya2019} for multiplicative noise. See also the articles \cite{Hairer2007,Rough,Hairer2013} for unique ergodicity results without a rate. To our best knowledge, none of the questions we ask here have been addressed so far. It appears they cannot be easily obtained by limiting arguments from the results on the finite-time, non-stationary density $\bar p_t(y_0;\cdot)$. New arguments are required: Our proofs are based on a combination of a disintegration formula for the stationary density and a careful analysis of the resulting conditioned evolution (with respect to the past of the Brownian motion before $0$). Instead of appealing to Malliavin calculus, we choose in this additive setting to base our proofs entirely
on a Girsanov-type representation of the density. More precisely, we
obtain a novel formula for the stationary density in terms of
the so-called \emph{Wiener-Liouville bridge} (see \cref{sec:Girsanov2}
for details). Combined with long-time stability arguments, a comprehensive study of the properties of the distribution of this Wiener-Liouville bridge allows us to prove the smoothness and Gaussian upper and lower bounds on the density of the stationary distribution. The argument for the smoothness is additionally based on a novel differentiability result for expectations of parameter-dependent Dol\'eans-Dade stochastic exponentials.

In the course of our work on \ref{it:q1}--\ref{it:q4} we realized that our sharp study of the conditional evolution is also useful for directly analyzing $\bar p_t(y_0; \cdot)$. More precisely, we are able to extend the results of \cite{Besalu2016} to a multi-dimensional setting in the small Hurst parameter regime $H<\frac12$ and to dramatically alleviate the regularity assumptions by considering drift coefficients which are merely H\"older continuous when $H>1/2$ and only measurable with at most linear growth when $H<1/2$. By also proving a Gaussian-type upper bound under the very same requirements, our work complements \cite{Baudoin2014,Baudoin2016} in which the authors considered the more general case of multiplicative noise, but had to impose the restrictive requirement of $\C_b^\infty$ coefficients.

\subsection{Main Results}

\paragraph{On the stationary density.} To obtain the regularity results for the stationary density, we impose the following conditions on the drift in \eqref{eq:sde}:

\begin{sk}\hypertarget{sk}{}
\item\label{cond:sk} $b\in\C^k(\R^n,\R^n)$ is Lipschitz continuous and contracting off the diagonal, \emph{i.e.}, there are constants $C,\kappa>0$ such that 
  \begin{equation}\label{eq:off_diagonal_contraction}
    \big\<b(y)-b(z),y-z\>\leq C-\kappa |y-z|^2 \qquad\forall\,y,z\in\R^n.
  \end{equation}
  In addition, if $k\geq 2$, we assume that 
  \begin{equation}\label{eq:growth_sk_1}
    \sup_{y\in\R^n}\sum_{i=2}^{k}\frac{\big|D^i b(y)\big|}{e^{C|y|}}<\infty,
  \end{equation}
  where $D^i$ denotes the total derivative operator of order $i\in\N$.
  Finally, if $H>\frac12$ and $k\geq 1$, we also assume that there is an $\alpha>1-\frac{1}{2H}$ such that
  \begin{equation}\label{eq:growth_sk_2}
     \sup_{|y|,|z|\leq R}\frac{\big|D^k b(y)- D^k b(z)\big|}{|y-z|^\alpha}\leq C e^{CR}\qquad\forall\,R>0.
   \end{equation} 
\end{sk}

The questions \ref{it:q1}--\ref{it:q3} are answered in the following theorem:
\begin{theorem}\label{thm:density}
  Let $H\in(0,1)$, $k\geq 0$, $\sigma\in\Lin{n}$ be non-degenerate, and $b:\R^n\rightarrow\R^n$ satisfy the condition \ref{cond:sk}. Then there is a unique stationary path space law $\prob_\pi$ to the equation \eqref{eq:sde} and its marginal $\pi$ has a Lebesgue density $p_\infty\in\C^k(\R^n)$. Furthermore, there are positive constants $c_1,c_2,C_1,C_2>0$ such that
  \begin{equation}\label{eq:gaussian_bounds}
    C_1 e^{-c_1|y|^2}\leq p_\infty(y) \leq C_2 e^{-c_2|y|^2} \quad\forall y\in\R^n,
  \end{equation}
  and, for any $\balpha\in \N_0^n$ with $|\balpha|\leq k$, there are positive constants $c_{\balpha},C_{\balpha}>0$  such that 
 \begin{equation}\label{eq:gaussian_lower_bound}
   \left|  \partial^{{\balpha}}_y p_\infty(y)\right|\leq C_{\balpha} e^{-c_{\balpha}|y|^2}  \qquad\forall\,y\in\R^n.
     \end{equation} 
\end{theorem}

The proof of \cref{thm:density} is achieved in \cref{sec:smoothness}.

\begin{remark}
  As mentioned earlier, \cref{thm:density} is based on the representation of the stationary density as average of a representation of the conditional density in terms of a Wiener-Liouville bridge as well as classical semimartingale techniques. Such results could be also obtained with Malliavin calculus. Nevertheless, it is worth noting that here, the regularity of the density is exactly the same as the one of $b$, whereas Malliavin calculus methods usually require more regularity for the drift (typically, to obtain a ${\cal C}^k$ density, one requires $b$ to be at least ${\cal C}^{k+\frac{n}{2}}$). Unfortunately, the bridge representation hinges on the Girsanov theorem, which in turn limits our method to additive noise. Extensions to multiplicative noise (probably by appealing to Malliavin calculus) are left to a future work.
\end{remark}

\paragraph{On the non-stationary density.} In our second main result, we
focus on non-stationary densities. As mentioned earlier, Gaussian-type
bounds on the non-stationary density can be also derived from our sharp
study of the conditional evolution. This allows us to extend
\cite[Theorem 1.2]{Besalu2016} to multi-dimensional systems with small
Hurst parameters and to considerably weaken the regularity requirements
on the drift vector field.  In the same way, the existing
upper-bounds of the literature (see \emph{e.g.} \cite{Baudoin2014,Baudoin2016}) are here obtained under optimal regularity assumptions in this specific additive setting. [The results of Besal\'u et al. and Baudoin et al. require $b\in\C_b^\infty(\R^n,\R^n)$.]

Let $\gamma\in(0,1]$. We say that $b:\R^n\to\R^n$ is $\gamma$-H\"older if
\begin{equation}\label{eq:local_holder}
  \|b\|_{\C^\gamma}\define\sup_{y,z\in\R^n}\frac{\big|b(y)-b(z)\big|}{|y-z|^\gamma}<\infty.
\end{equation}
We also write $\C^\gamma(\R^n,\R^n)$ for the space of $\gamma$-H\"older functions. We emphasize that these spaces do \emph{not} satisfy the usual inclusions $\C^\gamma(\R^n,\R^n)\subset\C^{\gamma^\prime}(\R^n,\R^n)$ for $\gamma^\prime<\gamma$.

\begin{cond_H}
  \item\label{cond:h} If $H>\frac12$, we assume that there is an $\alpha>1-\frac{1}{2H}$ such that $b\in\C^{\alpha}(\R^n,\R^n)$. If $H<\frac12$, $b:\R^n\to\R^n$ is Borel measurable with at most linear growth, \emph{i.e.}, there is a constant $C>0$ such that $\big|b(y)\big|\leq C\big(1+|y|\big)$ for all $y\in\R^n$.
\end{cond_H}

\begin{theorem}\label{thm:lower_bound_tindel}
  Consider the equation \eqref{eq:sde} with $\sigma\in\Lin{n}$ invertible and deterministic initial point $y_0\in\R^n$. If $b:\R^n\to\R^n$ satisfies assumption \ref{cond:h}, then \eqref{eq:sde} has a pathwise unique strong solution. In addition, the distribution of the solution at any positive finite time $T>0$ has a Lebesgue density $\bar{p}_t(y_0;.)$ and, for each $K\subset\R^n$ compact, there are constants $c_1,c_2,C_1,C_2>0$ (only $C_1,C_2$ depend on $K$) such that
  \begin{equation}\label{eq:lower_bound_tindel}
    \frac{C_1}{t^{nH}}\exp\left(-c_1\frac{|y-y_0|^2}{t^{2H}}\right)\leq \bar{p}_t(y_0;y)\leq\frac{C_2}{t^{nH}}\exp\left(-c_2\frac{|y-y_0|^2}{t^{2H}}\right) \qquad\forall\,y\in\R^n
  \end{equation}
  for all $y_0\in K$ and all $t\in(0,T]$.
\end{theorem} 

The well-posedness of \eqref{eq:sde} under assumption \ref{cond:h} is proven in \cref{prop:strong_well_posedness_tindel}, generalizing \cite{Catellier2016} to unbounded drifts. We believe that this requirement is in general optimal for well-posedness of \eqref{eq:sde}. The density estimates are established in \cref{sec:proof_tindel}.

\begin{remark} \leavevmode
  \begin{enumerate}
    \item It is straight-forward to adapt our method of proving \cref{thm:lower_bound_tindel} to allow a non-square diffusion matrix $\sigma\in\Lin[m]{n}$ acting on an $m$-dimensional fBm. Then the bounds \eqref{eq:lower_bound_tindel} hold with $n$ replaced by $m$:
    \begin{equation*}
      \frac{C_1}{t^{mH}}\exp\left(-c_1\frac{|y-y_0|^2}{t^{2H}}\right)\leq \bar{p}_t(y_0;y)\leq\frac{C_2}{t^{mH}}\exp\left(-c_2\frac{|y-y_0|^2}{t^{2H}}\right).
    \end{equation*}
    This slightly more general setting is studied in \cite{Besalu2016}. 

    \item The fact that the constants $C_1, C_2$ in \eqref{eq:lower_bound_tindel} depend on $y_0$ in general can be easily seen by choosing $b(y)=y$ in \eqref{eq:sde}. However, an easy adaptation of our proof shows that $C_1$ and $C_2$ can be chosen independently of $y_0$, provided that $b$ is bounded.

    \item Even though the lower bound is mainly a by-product of our study of the conditional evolution, it is worth noting that the upper bound requires a tailored argument. Actually, for the upper bound on the invariant density, we derive a corresponding estimate for the conditional evolution only at a small enough time $T_0>0$. In order to extend it to any $T>0$, we implement a bootstrapping argument with the help of a Chapman-Kolmogorov-type equation in this non-Markovian setting (see \cref{lem:fraction_chapman_kolmogorov} and \cref{sec:proof_tindel} for details).
  \end{enumerate}

\end{remark}

\paragraph{Smoothness of the stationary density in a parameter.} 

For $H=\f 12$, the invariant measure of \eqref{eq:parameter_sde} solves the elliptic equation $\mathscr{L}_\lambda^* \pi^\lambda=0$, where $\mathscr{L}_\lambda$ is the infinitesimal generator of $Y^\lambda$ and $\mathscr{L}_\lambda^*$ is its adjoint in $L^2(\R^n,\pi^\lambda)$. Therefore, the regularity of $\pi^\lambda$ follows from standard PDE theory, see \emph{e.g.} \cite{Pardoux2001,Pardoux2003,Pardoux2005,Veretennikov2011,Bogachev2015,Assaraf2018}. For Markov generators satisfying H\"ormander's condition, see also \cite{Li2018}. 

Since this approach fails for $H\neq\f 12$ and there is no explicit ansatz for the stationary measure allowing for a concrete analysis, it is not surprising that---to the best of our knowledge---there are no previous studies of the smoothness of the stationary measure in the parameter.  Before stating the result we first introduce parameter-dependent versions of assumption \ref{cond:sk}:

\begin{psk_loc}\hypertarget{psk_loc}{}
  \item\label{cond:psk_loc} We assume that, for each $K\subset\R^d$ compact, there are constants $C,\kappa>0$ such that the drift $b\in\C^k\big(\R^d\times\R^n,\R^n\big)$ satisfies all of the following:
  \begin{itemize}
    \item We have that 
    \begin{equation}\label{eq:local_uniform_lipschitz}
      \sup_{\lambda\in K}\sup_{y,z\in\R^n}\frac{\big|b(\lambda,y)-b(\lambda, z)\big|}{|y-z|}\leq C
    \end{equation}
    and
    \begin{equation}\label{eq:off_diagonal_contractive_lambda}
      \sup_{\lambda\in K}\bigl<b(\lambda,y)-b(\lambda,z),y-z\big\>\leq C-\kappa|y-z|^2, \qquad \forall\, y,z\in \R^n.
    \end{equation}
     
    \item If $k\geq 2$ the derivatives satisfy the estimate 
    \begin{equation}\label{eq:growth_bound}
      \sup_{\lambda\in K}\sup_{y\in\R^n}\sum_{i=2}^k\frac{\big|D_y^i b(\lambda,y)\big|}{e^{C|y|}}\leq C.
    \end{equation}
    
    \item In case of $H>\frac12$ and $k\geq 1$, there is additionally a number $\alpha>1-\frac{1}{2H}$ such that
    \begin{equation}\label{eq:holder_bound}
      \sup_{\lambda\in K}\sup_{|y|,|z|\leq R}\frac{\big|D^k_y b(\lambda,y)- D_y^k b(\lambda,z)\big|}{|y-z|^\alpha}\leq C e^{C R}\qquad\forall\,R>0.
    \end{equation}
  \end{itemize}
\end{psk_loc}

\begin{psk}\hypertarget{psk}{}
  \item\label{cond:psk} $b:\R^d\times\R^n\to\R^n$ falls in the regime of \ref{cond:psk_loc}, but \eqref{eq:local_uniform_lipschitz}--\eqref{eq:holder_bound} hold for constants $C,\kappa>0$ that can be chosen uniformly in $\lambda\in\R^d$.
\end{psk}

\begin{theorem}\label{thm:smooth}
  Let $H\in(0,1)$, $k\geq 0$, $\sigma\in\Lin{n}$ be non-degenerate, and $b:\R^d\times\R^n\to\R^n$ satisfy \ref{cond:psk_loc}. Then, for each $\lambda\in\R^d$, \eqref{eq:parameter_sde} has a unique stationary measure $\pi^\lambda$,  which also has a Lebesgue density $p_\infty^\lambda$. In addition, the mapping $(\lambda, y)\mapsto p_\infty^\lambda(y)$ is in $\C^k$. 

  If $b$ further satisfies \ref{cond:psk}, then $p_\infty^\lambda\in\C_b^k(\R^d\times\R^n)$.
\end{theorem} 

The parameter-dependent setting of \cref{thm:smooth} makes it difficult to prove a comparable result with the tools of Malliavin calculus. In fact, this would require an integration by parts formula not only in the spatial coordinate, but also with respect to the parameter. To the best of our knowledge such a result is not available so far.

\paragraph{Application: The fractional averaging principle revisited.} 

While \cref{thm:smooth} is interesting in its own right, our original motivation for studying the parameter-dependent equation \eqref{eq:parameter_sde} stems from our recent work on the $\varepsilon\to 0$ of the \emph{slow-fast system}
\begin{alignat}{4}
  dX_t^\varepsilon&=f(X_t^\varepsilon,Y_t^\varepsilon)\,dt+g(X_t^\varepsilon,Y_t^\varepsilon)\,dB_t, &\qquad X_0^\varepsilon&=X_0, \label{eq:slow}\\
  dY_t^\varepsilon &= \frac{1}{\varepsilon}b(X_t^\varepsilon,Y_t^\varepsilon)\,dt+\frac{1}{\varepsilon^{\hat H}}\sigma\,d\hat B_t, &\qquad Y_0^\varepsilon&=Y_0, \label{eq:fast}
\end{alignat}
where $B$ and $\hat B$ are independent fBms with Hurst parameters $H\in\big(\frac12,1\big)$ and $\hat{H}\in\big(1-H,1\big)$, respectively. Let us suppose that, for each $x\in\R^d$, the equation
\begin{equation}\label{eq:fixed_x_fast}
  dY_t^x=b(x,Y_t^x)\,dt+\sigma\,d\hat{B}_t
\end{equation}
has unique stationary measure $\pi^x\in\P(\R^n)$. Under a set of conditions, it was shown in \cite{Li2020} that $X^\varepsilon\to\bar{X}$ in H\"older norm in probability, where the limit solves the \emph{effective equation}
\begin{equation}\label{eq:effective_equation}
  d\bar{X}_t=\bar{f}(\bar{X}_t)\,dt+\bar{g}(\bar{X}_t)\,dB_t.
\end{equation}
Here, $\bar{f}(x)=\int_{\R^n} f(x,y)\,\pi^x(dy)$ and \emph{mutatis mutandis} for $\bar{g}$. To obtain the convergence $X^\varepsilon\to\bar{X}$, \cite[Theorem 1.2]{Li2020} had to impose the ad-hoc assumption $\bar{g}\in\C_b^2\big(\R^d,\Lin[m]{d}\big)$ to ensure the well-posedness of \eqref{eq:effective_equation}. That work verified that this regularity condition is satisfied for $g\in\C_b^3\big(\R^d\times\R^n,\Lin[m]{d}\big)$ and a \emph{uniformly convex} (or \emph{strongly contractive}) drift vector fields. This is to say that there is a $\kappa>0$ such that $\braket{b(x,y)-b(x,z),y-z}\leq -\kappa|y-z|^2$ for all $x\in\R^d$ and all $y,z\in\R^n$.

\Cref{thm:smooth} allows us remove the ad-hoc smoothness assumption since it is a consequence of the other conditions. To state the improved result we introduce the following class of drifts:
\begin{definition}\label{define-semi-contractive}
  Let $\xi, R\geq 0$ and $\kappa>0$. We write $\S(\kappa, R, \xi)$ for the set of functions $b:\R^d\times\R^n\to\R^n$ satisfying all of the following conditions:
  \begin{itemize}
    \item $b$ falls in the regime of \hyperlink{psk}{$\mathbf{(PS)}_{2,H}$}.
    \item There is a $C>0$ such that
    \begin{equation*}
      \big|b(x,y)\big|\leq C\big(1+|x|+|y|\big)\qquad\forall\,x\in\R^d,\,y\in\R^n.
    \end{equation*}
    \item For each $R>0$ there is an $L_R>0$ such that 
    \begin{equation*}
      \sup_{y\in\R^n}\big|b(x_1,y)-b(x_2,y)\big|\leq L_R|x_1-x_2|\qquad\forall\,|x_1|, |x_2|\leq R.
    \end{equation*}
    \item It holds that
    \begin{equation*}
    \sup_{x\in\R^d}\big\langle b(x,y_1)-b(x,y_2),y_1-y_2\big\rangle\leq\begin{cases}
      -\kappa|y_1-y_2|^2, & |y_1|,|y_2|\geq R,\\
      \xi|y_1-y_2|^2, &\text{otherwise}.\\
  \end{cases}
  \end{equation*}

  \end{itemize}
  
\end{definition}

\begin{corollary}\label{cor:averaging}
  Consider the slow-fast system \eqref{eq:slow}--\eqref{eq:fast} with $X_0\in L^\infty(\Omega)$, $Y_0\in\bigcap_{p\geq 1} L^p(\Omega)$, and suppose that $f:\R^d\times\R^n\to\R^n$ is Lipschitz continuous as well as $g:\R^d\times\R^n\to\Lin[m]{n}$ is of class $\C_b^2$. Let $\alpha<H$. Then, for each $\kappa>0$, there exists a number $\Xi=\Xi(\alpha,\kappa,R)>0$ such that, whenever $b\in \S(\kappa, R, \Xi)$ for each $x\in\R^d$, all of the following hold: 
  \begin{itemize}
    \item For each $x\in\R^d$ the equation \eqref{eq:fixed_x_fast} has a unique stationary path space law $\prob_{\pi^x}$.
    \item For each $T>0$ the solution $X^\varepsilon$ to \eqref{eq:slow} converges to the unique solution $\bar{X}$ of \eqref{eq:effective_equation} in $\C^\alpha\big([0,T],\R^d\big)$ in probability as $\varepsilon\to 0$.
  \end{itemize}
\end{corollary}
\begin{proof}
  The corollary follows from \cite[Theorem 1.2]{Li2020} since 
  \begin{equation*}
     x\mapsto \bar{g}(x)=\int_{\R^n} g(x, y) p_\infty^x(y)\,dy\in\C_b^2\big(\R^d,\Lin[m]{d}\big)
  \end{equation*}
  by \cref{thm:smooth}.
\end{proof}

\paragraph{Organization of the article.} In \cref{sec:preliminaries} we explain how to build an auxiliary Markov process for \eqref{eq:sde}, collect some results from the theory of fractional integrals and derivatives, and prove a number of preliminary results. In \cref{sec:Girsanov1} we apply the Girsanov theorem to the Liouville process and establish well-posedness of \eqref{eq:sde} under the condition \ref{cond:h} as well as a representation of the conditional density. The latter is then written as an expectation over a so-called Wiener-Liouville bridge in \cref{sec:Girsanov2}, which allows us to prove Gaussian-type bounds on the conditional density and its derivatives. Finally, \cref{sec:proofmain} deduces the main results of the article.

\paragraph{Acknowledgements.}  We thank the three anonymous referees for their helpful comments.

\subsection{Notation} 

We mainly use standard notation: $|\cdot|$ denotes the Euclidean norm (note that the dimension may vary though), $\braket{\cdot,\cdot}$ is the Euclidean scalar product, and $\P(\X)$ is the set of Borel probability measures on a Polish space $\X$. The law of a random variable $X$ is abbreviated by $\L(X)=\prob(X\in\cdot)$. The notation $\C_0\big([0,T],\R^n\big)$ designates the space of continuous functions from $[0, T]$ to $\R^n$ vanishing at $0$. This space becomes a separable Banach space when equipped with the norm $\|f\|_\infty\define\sup_{t\in[0,T]}|f(t)|$. The space of locally H\"older continuous functions $f:\R_+\to\R^n$ of order $\gamma\in(0,1]$ is denoted by $\Cloc^\gamma(\R_+,\R^n)$ and we set $\Cloc^{\gamma-}(\R_+,\R^n)\define\bigcap_{\gamma^\prime\in(0,\gamma)}\Cloc^{\gamma^\prime}(\R_+,\R^n)$ and \emph{mutatis mutandis} for $\Cloc^{\gamma+}$. Recall that $f\in\Cloc^{\gamma}(\R_+,\R^n)$ if $f\restriction_{[0,T]}\in\C^\gamma\big([0,T],\R^n\big)$ for each $T>0$. For $\alpha=k+\gamma$, $k\in\N_0$, $\gamma\in(0,1]$, we understand $\C^\alpha$ as the space of $k$-times differentiable function whose highest order derivative satisfies \eqref{eq:local_holder} with exponent $\gamma$. The H\"older norm is defined by 
\begin{equation*}
  \vertiii{f}_{\C^\alpha}\define\sum_{i=0}^k\big\|D^i f\big\|_\infty+\big\|D^k f\big\|_{\C^\gamma}.
\end{equation*}
{where $\|\,.\|_{\C^\gamma}$ is defined by \eqref{eq:local_holder}.} We write $\C^k$ (resp. $\C^k_b$) for the space of (bounded) $k$-times differentiable functions with (bounded) continuous derivatives. Let $\balpha,\bbeta\in\N_0^n$ be multi-indices. It is customary to write $\balpha\preccurlyeq\bbeta$ if $\balpha_i\leq\bbeta_i$ for each $i=1,\dots,n$. We understand $|\balpha|=\balpha_1+\cdots+\balpha_n$ and, for $g:\R^n\to\R^m$, the notation $\partial^{\balpha}g=\partial_{y_1}^{\balpha_1}\cdots\partial_{y_n}^{\balpha_n}g$ designates the partial derivative of order $\balpha$. Occasionally, we amend these operators by a subscript, \emph{e.g.}, $\partial_y^{\balpha} h(\lambda, y)$ for $h:\R^d\times\R^n\to\R^m$ in order to emphasize the variables on which they act. We let $D^k g(y)\in L\big((\R^n)^{\otimes k},\R^m\big)$ be the $k^{\textup{th}}$ (total) derivative of $g:\R^n\to\R^m$.

Furthermore, $[\mu]_{\X}$ and  $[\mu]_{\Y}$ are the marginals of a measure $\mu\in\P\big(\X\times\Y\big)$. The total variation norm of a measure $\nu\in\P(\X)$ is denoted by $\TV{\nu}=\sup_{A\in\B(\X)}\big|\nu(A)\big|$, where $\B(\X)$ denotes the Borel $\sigma$-field on $\X$. We write $\mu\approx\nu$ if the measures are equivalent, that is, they are absolutely continuous with respect to each other; in symbols $\mu\ll\nu$ and $\nu\ll\mu$.

\section{Preliminaries}\label{sec:preliminaries}

\subsection{Stochastic Dynamical Systems Revisited}\label{sec:invariant_measures}

This section explains the connection between stationary path space laws of the equation \eqref{eq:sde} and the invariant measures of an auxilliary Feller process. To this end, we first recall a standard representation of fBm as an integral of an $n$-dimensional, two-sided standard Wiener process $(W_t)_{t\in\R}$  which will become important in the sequel: In \cite{Mandelbrot1968} Mandelbrot and van Ness showed that there is a normalization constant $\alpha_H>0$ such that 
\begin{equation}\label{eq:mandelbrot}
  B_t=\alpha_H\int_{-\infty}^0\Big((t-u)^{H-\frac12}-(-u)^{H-\frac12}\Big)\,dW_u+\alpha_H\int_0^t (t-u)^{H-\frac12}\,dW_u,\qquad t\geq 0,
\end{equation}
is an fBm with Hurst parameter $H\in(0,1)$. We understand the integrand to be multiplied by the $n\times n$ identity matrix. This representation immediately furnishes the following \emph{locally independent decomposition} of the fractional Brownian increments: For $h,t\geq 0$ we have
\begin{align}
  B_{t+h}-B_t&=\alpha_H\int_{-\infty}^t\Big((t+h-u)^{H-\frac12}-(t-u)^{H-\frac12}\Big)\,dW_u+\alpha_H\int_t^{t+h} (t+h-u)^{H-\frac12}\,dW_u \nonumber\\
  &\define \bar{B}_h^t + \tilde{B}_h^t. \label{eq:innovation_process}
\end{align} 
If $t=0$ we also adopt the short-hands $\bar B$ and $\tilde{B}$. We call $\bar{B}^t$ the \emph{history} and $\tilde{B}^t$ the \emph{innovation process}. The latter is oftentimes referred to as \textit{Liouville process} in the literature (in reference to the Riemann-Liouville integral defined in \eqref{eq:fractional_integral}) and has a long history (see \emph{e.g.} \cite{Levy, Lifshits06}). It agrees with a standard Wiener process when $H=\f 12$.

It is clear that the evolution \eqref{eq:sde} is not Markovian, whence the usual definitions of stationary and invariant measures through the transition semigroup do not apply. However, it was observed by Hairer \cite{Hairer2005} that the joint evolution $\mathfrak{Z}_t\define\big(Y_t,(W_{s+t})_{s\leq 0}\big)$, $t\geq 0$, is indeed Markovian. His approach differed from the standard path space construction in the theory of random dynamical systems \cite{Arnold2013} as it only allows for `physical' stationary solutions. The state space of $(\mathfrak{Z}_t)_{t\geq 0}$ is $\R^n\times\cB_H$, where $\cB_H$ is a separable Banach space of functions $w:\R_-\to\R^n$ on which the (left-sided) Wiener measure $\sW$ is supported and which renders the application $W\mapsto B$ in \eqref{eq:mandelbrot} continuous. More specifically, let $\cC_0^\infty(\R_-,\R^n)$ be the space of smooth compactly supported functions $w:\R_-\to\R^n$ with $w(0)=0$. The space $\cB_H$ is defined as the closure of $\cC_0^\infty(\R_-,\R^n)$ in the norm
\begin{equation*}
  \|w\|_{\cB_H}\define\sup_{s,t\leq 0}\frac{\big|w(t)-w(s)\big|}{\sqrt{1+|t|+|s|}|t-s|^{\frac{1-H}{2}}}.
\end{equation*} In particular, we also understand the initial condition to \eqref{eq:sde} as an initial condition for $(\mathfrak{Z}_t)_{t\geq 0}$:
\begin{definition}\label{def:invariant_measure}
  \leavevmode
  \begin{enumerate}
    \item A probability measure $\mu\in\P\big(\R^n\times\cB_H\big)$ is called a \emph{generalized initial condition} for \eqref{eq:sde} if $[\mu]_{\cB_H}=\sW$.
   \item A generalized initial condition $\I_\pi$ is called an \emph{invariant measure} of \eqref{eq:sde} if it is invariant for the transition semigroup of $(\mathfrak{Z}_t)_{t\geq 0}$. We write $\pi\define[\I_\pi]_{\R^n}$ for its first marginal.
  \end{enumerate}
\end{definition}

If the diffusion coefficient $\sigma\in\Lin{n}$ is non-degenerate, it can also be shown that $(Y_t)_{t\geq 0}$ is strictly stationary if and only if the equation \eqref{eq:sde} is started from an invariant measure of the auxiliary process $(\mathfrak{Z}_t)_{t\geq 0}$, which renders the construction meaningful, see \cite[Propositions 2.17 and 2.18]{Hairer2005}. In general, it may be possible that an invariant measure $\I_\pi$ is not uniquely characterized by its $\R^n$-marginal, whence our notation were ill-defined. However, under the assumptions of this article there is anyways a unique invariant measure so that no confusion can arise:

\begin{proposition}[{\cite[Theorems 1.2 and 1.3]{Hairer2005}}]\label{prop:tv_convergence_invariant_measure}
  Let $b:\R^n\to\R^n$ be Lipschitz continuous. If there are $C,\kappa>0$ such that $\braket{b(y)-b(z),y-z}\leq C-\kappa|y-z|^2$ for all $y,z\in\R^n$, then there is an invariant measure $\I_\pi$ (in the sense of \cref{def:invariant_measure}) for the equation \eqref{eq:sde}. In addition, for each generalized initial condition $\mu\in\P(\R^n\times\cB_H)$ with $\int_{\R^n}|y|\,[\mu]_{\R^n}(dy)<\infty$, the solution $Y$ started in $\mu$ satisfies
  \begin{equation*}
     \lim_{t\to\infty}\tv{\L(Y_t)-\pi}=0.
  \end{equation*} 
\end{proposition}

\begin{remark}
  There is another widely popular Volterra integral representation of fBm:
  \begin{equation}\label{eq:decreusefond_representation}  
    B_t=\int_0^t K_H(t,s)\,dW_s,\qquad t\geq 0,
  \end{equation} 
  for a suitable integral kernel $K_H$, see \cite{Decreusefond-Usutunel} for details. This representation is usually used to define the Malliavin derivative with respect to $(B_t)_{t\geq 0}$. Consequently, (almost) all of the works studying the density of $dX_t=b(X_t)\,dt+\sigma(X_t)\,dW_t$ in the \emph{non-stationary regime} are based on \eqref{eq:decreusefond_representation}. However, this representation does not give rise to a \emph{stationary noise process} as defined in \cite[Definition 2.6]{Hairer2005}, whence making the study of stationary measures intractable.  
\end{remark}

Following \cite[Notation 4.1]{Deya2019}, for $\gamma\in(0,1)$ and $k\in\N$ we set
\begin{equation}\label{eq:defmathfrakekgam}
  \mathfrak{E}_\gamma^k\big([0,T],\R^n\big)\define\Big\{f\in\C_0\big([0,T],\R^n\big)\cap\C^k\big((0,T],\R^n\big):\,\|f\|_{\mathfrak{E}_\gamma^k} <\infty \Big\}, 
\end{equation}
where
\begin{equation*}
  \|f\|_{\mathfrak{E}_\gamma^k}\define\sum_{j=1}^k\sup_{t\in(0,T]}t^{j-\gamma}\big|D^{j} f(t)\big|.
\end{equation*}
By \cite[Lemma 6.5(i)]{Deya2019}, for each $\gamma\in(0,H)$, the operator $\cP^H:\C_0^\infty(\R_-,\R^n)\to\mathfrak{E}_\gamma^2\big([0,T],\R^n\big)$ defined by
\begin{equation}\label{PH}
  \cP^H w(t)\define \alpha_H\int_{-\infty}^0\left((t-u)^{H-\frac12}-(-u)^{H-\frac12}\right)\dot{w}(u)\,du,\qquad t\in(0,T],
\end{equation}
and $\cP^Hw(0)=0$ is continuous and consequently extends to a unique bounded linear application on $\cB_H$. 

For $\ell\in\Cloc^{H-}(\R_+,\R^n)$, we will denote by $\Phi_t(\ell)$ a  solution to
\begin{equation}\label{eq:innovation_sde}
  \Phi_t(\ell)=\ell(t)+\int_0^t b\big(\Phi_s(\ell)\big)\,ds+\sigma\tilde{B}_t,\qquad t\geq 0.
\end{equation}
In the sequel, we will only need weak existence and uniqueness in law for \eqref{eq:innovation_sde} for any $\ell\in\Cloc^{H-}(\R_+,\R^n)$. Such a result will be provided in \cref{lem:innovation_sde_well_posed} under Assumption \ref{cond:h}. At this stage, we shall instead record the following important observation: 

\begin{proposition}\label{prop:integral_density}
  Let $\mu\in\P(\R^n\times \cB_H)$ be a generalized initial condition of \eqref{eq:sde}. Then, for each Borel set $A\in\B(\R^n)$ and each $T>0$, we have
  \begin{equation}\label{eq:general_integral}
    \prob\big(Y_T\in A\big)=\int_{\R^n\times\cB_H}\prob\big(\Phi_T(z+\sigma\cP^H w)\in A\big)\,\mu(dz,dw).
  \end{equation}
  In particular if, for each $(z,w)\in\R^n\times\cB_H$, $\L\big(\Phi_T(z+\sigma\cP^H w)\big)$ has a density with respect to the $n$-dimensional Lebesgue measure $\lambda^n$, then $[\mu]_{\R^n}\ll\lambda^n$ and
  \begin{equation}\label{eq:density_integral}
    p_{Y_T}(y)=\int_{\R^n\times\cB_H} p_T(z,w;y)\,\mu(dz,dw)
  \end{equation}
  for Lebesgue-a.e. $y\in\R^n$, where $p_{Y_T}\define\frac{d\L(Y_T)}{d\lambda^n}$ and $p_T(z,w;\cdot)\define\frac{d\L\big(\Phi_T(z+\sigma\cP^H w)\big)}{d\lambda^n}$.

  \emph{A fortiori}, the stationary density satisfies 
  \begin{equation*}
    p_\infty(y)=\int_{\R^n\times\cB_H}p_T(z,w;y)\,\I_\pi(dz,dw)
  \end{equation*}
  for each $T>0$.
\end{proposition}
\begin{proof}
  Let $(\mathfrak{Z}_t)_{t\geq 0}$ be the auxiliary Markov process constructed above. Then
  \begin{align*}
    \prob\big(Y_T\in A\big)=\prob(\mathfrak{Z}_T\in A\times \cB_H)&=\int_{\R^n\times \cB_H} \prob\big(\mathfrak{Z}_T\in A\times \cB_H \,|\,\mathfrak{Z}_0=(z, w)\big) \, \mu(dz,dw)\\
   &=\int_{\R^n\times \cB_H} \prob\big(Y_T\in A \,|\,\mathfrak{Z}_0=(z, w)\big)\, \mu(dz,dw)\\
   &=\int_{\R^n\times \cB_H} \prob\big(\Phi_T(z+\sigma\cP^H w)\in A\big)\, \mu(dz,dw),
  \end{align*}
  where the last step follows from the construction of \cite{Hairer2005}.
\end{proof}

Albeit being not Markovian, the density of the solution to \eqref{eq:sde} still satisfies a modified Chapman-Kolmogorov relation. This property will be of primal importance in the proof of \cref{thm:lower_bound_tindel}, see in particular \cref{sec:proof_tindel}.

\begin{lemma}[Fractional Chapman-Kolmogorov equation]\label{lem:fraction_chapman_kolmogorov}
  Let $s,t>0$. Suppose that the Lebesgue density $p_s(\ell;\cdot)$ of $\Phi_s(\ell)$ exists for each $\ell\in\Cloc^{H-}(\R_+,\R^n)$. In the notation of \cref{thm:lower_bound_tindel} it then holds that
  \begin{equation*}
    \bar{p}_{t+s}(y_0;y)=\Expec{p_s\Big(Y_t+\sigma\bar{B}^t; y\Big)} \qquad\forall\,y_0,y\in\R^n,
  \end{equation*}
  where $\bar{B}^t$ was defined in \eqref{eq:innovation_process} and $Y$ is the solution to \eqref{eq:sde} with generalized initial condition $\delta_{y_0}\otimes\sW$.
\end{lemma}

\begin{proof}
  Let $f:\R^n\to\R_+$ be a Borel function and let $(\F_t)_{t\geq 0}$ be the filtration generated by the Wiener process driving $B$ through \eqref{eq:mandelbrot}. Then we can write
  \begin{equation*}
    Y_{t+s}=Y_t+\int_0^s b\big(Y_{t+r}\big)\,dr+\sigma\big(\bar{B}_s^t+\tilde{B}_s^t\big)
  \end{equation*}
  and, since $\tilde{B}^t\overset{d}{=}\tilde{B}$ for each $t\geq 0$, it follows that
  \begin{equation*}
    \Expec{f\big(Y_{t+s}\big)\,\middle|\,\F_t}=\Expec{f\big(\Phi_s(\ell)\big)}\Big|_{\ell=Y_t+\sigma\bar{B}^t}=\int_{\R^n} f(y) p_{s}\Big(Y_t+\sigma\bar{B}^t; y\Big)\,dy 
  \end{equation*}
  and the required result follows.
\end{proof}

\subsection{Uniform Gaussian Tails of the Invariant Distribution}\label{sec:tails_invariant_measure}

As a first step toward the Gaussian bounds for the stationary density $p^\lambda_\infty$ of the equation \eqref{eq:parameter_sde},  we show that, under the conditions \hyperlink{psk_loc}{$\mathbf{(PS_{\mathrm{loc}})}_{0,H}$} and \hyperlink{psk}{$\mathbf{(PS)}_{0,H}$}, $\I_{\pi^\lambda}$ has Gaussian tails (locally) uniform in $\lambda$, that is, for $\rho>0$ sufficiently small,
\begin{equation}\label{eq:gaussian_tails_invariant_measure}
  \sup_{\lambda\in K}\int_{\R^n\times\cB_H} \exp\Big(\rho\big(|z|^2+\|w\|_{\cB_H}^2\big)\Big)\,\I_{\pi^\lambda}(dz,dw)<\infty,
\end{equation}
either for $K\subset\R^d$ compact or $K=\R^d$. Such a result then of course also applies in the setting of \cref{thm:density}. 

We first recall the following quantitative version of Fernique's celebrated theorem \cite{Fernique1970}:

\begin{proposition}[{\cite[Theorem 2.8.5]{Bogachev1998}}]\label{prop:fernique}
  Let $\mu$ be a centered Gaussian measure on a separable Banach space $\big(\X,\|\cdot\|_{\X}\big)$. Set
  \begin{equation*}
    \zeta\define\inf\left\{t>0:\,\mu\big(\|x\|_\X\leq t)=\frac34\right\}.
  \end{equation*}
  Then there is a constant $C(\zeta)$ depending only on $\zeta$ such that $\zeta\mapsto C(\zeta)$ is increasing and, for each $\rho\leq\frac{1}{24\zeta^2}\log(3)$,
  \begin{equation*}
    \int_{\X}\exp\big(\rho\|x\|^2_\X\big)\,\mu(dx)\leq C(\zeta).
  \end{equation*}
\end{proposition}

The proof of \eqref{eq:gaussian_tails_invariant_measure} is obtained by comparing the solution to \eqref{eq:sde} to the fractional Ornstein-Uhlenbeck process (fOU) with zero initial condition \cite{Cheridito2003}:
\begin{equation}\label{eq:fou_sde}
  dZ_t=-Z_t\,dt+\sigma\,dB_t, \qquad Z_0=0.
\end{equation} 
For $\kappa,T>0$ we define
\begin{equation*}
  \cN_{\kappa}\big([0,T],\R^n\big)=\left\{f:[0,T]\to\R^n:\,\|f\|_{\cN_{\kappa}}=\sqrt{\int_0^Te^{-\kappa(T-s)}\big|f(s)\big|^2\,ds}<\infty\right\}.
\end{equation*}
Notice that $e^{-\frac{\kappa}{2} T}\|\cdot\|_{L^2}\leq \|\cdot\|_{\cN_\kappa}\leq\|\cdot\|_{L^2}$, so that $\cN_\kappa\big([0,T],\R^n\big)$ is \emph{a fortiori} a separable Banach space.
\begin{lemma}\label{lem:norm_fou_tail}
For each $\kappa>0$ there are $c,C>0$ such that
  \begin{equation*}
    \sup_{T\geq 0}\prob\big(\|Z\|_{\cN_\kappa([0,T])}>|z|\big)\leq C e^{-c|z|^2}\qquad\forall\, z\in\R^n.
  \end{equation*}
\end{lemma}
\begin{proof}
  The solution of \eqref{eq:fou_sde} is given by 
  \begin{equation*}
    Z_t=\sigma\int_0^t e^{-(t-s)}\,dB_s, \qquad t\geq 0,
  \end{equation*}
  where the integration is in the Paley-Wiener sense (or in this case equivalently as a Young integral). One checks that $Z_t\sim N(0,\Sigma_t)$, where
  \begin{equation*}
    \Sigma_t\define 2H e^{-t}\sigma\sigma^{\top}\int_0^t s^{2H-1}\cosh(t-s)\,ds.
  \end{equation*}
  In particular, we observe that $\sup_{t\geq 0}|\Sigma_t|<\infty$. Consequently, there are $c,C>0$ such that
  \begin{equation}\label{eq:fou_tail}
    \sup_{t\geq 0}\prob\big(|Z_t|>|z|\big)\leq C e^{-c|z|^2}\qquad\forall\,z\in\R^n.
  \end{equation}
  It is also clear that the sample paths of $Z_t$ lie in $\cN_\kappa\big([0,T],\R^n\big)$ for each $\kappa>0$:
    \begin{equation*}
      \Expec{\|Z\|_{\cN_\kappa([0,T])}^2}=\int_0^T e^{-\kappa(T-s)}\Expec{|Z_s|^2}\,ds\leq\frac{1}{\kappa}\sup_{t\geq 0}\Expec{|Z_t|^2}, 
    \end{equation*}
  whence for each $\kappa,T>0$, the law of $(Z_t)_{t\in [0,T]}$ defines a centered Gaussian measure on the separable Banach space $\cN_\kappa\big([0,T],\R^n\big)$. Since $\sup_{T\geq 0}\Expec{\|Z\|_{\cN_\kappa([0,T])}}<\infty$ the required tail estimate follows from \Cref{prop:fernique}.
\end{proof}

As a next step, we show that the distance of the solution to \eqref{eq:sde} and the fOU can be pathwise controlled:

\begin{lemma}\label{lem:solution_stability}
  Let $H\in(0,1)$ and $(Y_t^\lambda)_{t\geq 0}$ be the solution to \eqref{eq:parameter_sde} started from the generalized initial condition $\delta_0\otimes\sW$. Assume that the drift $b$ satisfies \hyperlink{psk_loc}{$\mathbf{(PS_{\mathrm{loc}})}_{0,H}$}. Then, for each $K\subset\R^d$ compact, there is a $C_K>0$ such that
  \begin{equation*}
    \sup_{\lambda\in K}\big|Y_T^\lambda-Z_T\big|\leq C_K\big(1+\|Z\|_{\cN_\kappa([0,T])}\big) \qquad\forall\,T\geq 0.
  \end{equation*}
  If $b$ satisfies even \hyperlink{psk}{$\mathbf{(PS)}_{0,H}$}, then the constant $C_K>0$ can be chosen uniformly over $\lambda\in\R^d$.
\end{lemma}
\begin{proof}
  Fix $K\subset\R^d$ compact and assume \hyperlink{psk_loc}{$\mathbf{(PS_{\mathrm{loc}})}_{0,H}$}. By Young's inequality we have
  \begin{align}
    \frac{d}{dt}\big|Y_t^\lambda-Z_t\big|^2&=2\big<b(\lambda, Y_t^{\lambda})+Z_t,Y_t^\lambda-Z_t\big> \nonumber\\
    &\leq 2\left(C-\kappa\big|Y_t^\lambda-Z_t\big|^2+\big|b(\lambda,Z_t)+Z_t\big|\big|Y_t^\lambda-Z_t\big|\right) \label{eq:stability_derivative}\\
    &\leq 2C-\kappa\big|Y_t^\lambda-Z_t\big|^2+\frac{2}{\kappa}|Z_t|^2\qquad\forall\,\lambda\in K. \nonumber
  \end{align}
  Hence, the function $f(t)\define e^{\kappa t}\big|Y_t^\lambda-Z_t\big|^2$ satisfies the differential inequality
  \begin{equation*}
    \frac{d}{dt}f(t)\leq e^{\kappa t}\left(2C+\frac{2}{\kappa}|Z_t|^2\right),
  \end{equation*}
  so that
  \begin{equation*}
    \sup_{\lambda\in K}\big|Y_t^\lambda-Z_t\big|^2\lesssim \int_0^t e^{-\kappa(t-s)}\big(1+|Z_s|^2\big)\,ds\qquad\forall\,t\in[0,T].
  \end{equation*}
  The claimed estimate follows at once. 

  The argument in case $b$ satisfies \hyperlink{psk}{$\mathbf{(PS)}_{0,H}$} is similar. 
\end{proof}

Finally we can prove \eqref{eq:gaussian_tails_invariant_measure}:
\begin{corollary}\label{cor:invariant_measure}
  Let $H\in(0,1)$ and consider \eqref{eq:parameter_sde} with a drift falling in the regime of \hyperlink{psk_loc}{$\mathbf{(PS_{\mathrm{loc}})}_{0,H}$}. Then, for each $K\subset\R^d$, there is a $\rho>0$ such that the invariant measure satisfies
  \begin{equation*}
    \sup_{\lambda\in K}\int_{\R^n\times\cB_H} \exp\Big(\rho\big(|z|^2+\|w\|_{\cB_H}^2\big)\Big)\,\I_{\pi^\lambda}(dz,dw)<\infty.
  \end{equation*}
  If $b$ satisfies \hyperlink{psk}{$\mathbf{(PS)}_{0,H}$}, then $\rho$ can be chosen uniform over $\lambda\in\R^d$.
\end{corollary}
\begin{proof}
  Fix $K\subset\R^d$ compact and assume \hyperlink{psk_loc}{$\mathbf{(PS_{\mathrm{loc}})}_{0,H}$}. By Young's inequality, it is enough to show that, for a positive number $\rho$, 
  \begin{equation*}
    \sup_{\lambda\in K}\int_{\R^n} e^{\rho|z|^2}\,\pi^\lambda(dz)+\int_{\cB_H} e^{\rho\|w\|_{\cB_H}^2}\,\sW(dw)<\infty.
  \end{equation*}
  The fact that the second integral is finite for sufficiently small $\rho$ is a direct consequence of Fernique's theorem and standard properties of the Wiener measure, see \emph{e.g.} \cite[Lemma 3.8]{Hairer2005} for details. To bound the first integral we claim that, for some constants  $c,C>0$,
  \begin{equation}\label{eq:uniform_gaussian}
    \sup_{\substack{t\geq 0\\ \lambda\in K}}\prob\big(|Y_t^\lambda|>|z|\big)\leq C e^{-c|z|^2} \qquad\forall\,z\in\R^n,
  \end{equation}
  where $Y^\lambda$ solves \eqref{eq:parameter_sde} with generalized initial condition $\delta_0\otimes\sW$. By \cref{prop:tv_convergence_invariant_measure} we know that $\L(Y_t^\lambda)\to\pi^\lambda$ as $t\to\infty$ in total variation distance. Hence, we \emph{a forteriori} have
  \begin{equation*}
    \sup_{\lambda\in K}\pi^\lambda\big(\{y\in\R^n:\,|y|>|z|\}\big)=\sup_{\lambda\in K}\lim_{t\to\infty}\prob\big(|Y_t^\lambda|>|z|\big)\leq C e^{-c|z|^2}\qquad\forall\,z\in\R^n
  \end{equation*}
  from \eqref{eq:uniform_gaussian}.
  It follows that
  \begin{equation*}
    \int_{\R^n}e^{\rho|z|^2}\,\pi^\lambda(dz)\leq C\left(1+\int_1^\infty\pi^\lambda\left(\left\{y\in\R^n:\,|y|>\rho^{-\frac12}\sqrt{\log s}\right\}\right)\,ds\right)<\infty
  \end{equation*}
  for $\rho<c$. 

  To see \eqref{eq:uniform_gaussian}, we split
  \begin{equation*}
    \prob\big(|Y_t^\lambda|>|z|\big)\leq\prob\left(\big|Y_t^\lambda-Z_t\big|>\frac{|z|}{2}\right)+\prob\left(|Z_t|>\frac{|z|}{2}\right).
  \end{equation*}
  We have already seen in \eqref{eq:fou_tail} that the fractional Ornstein-Uhlenbeck process has a Gaussian tail uniformly in time. By \cref{lem:solution_stability}, we obtain
  \begin{equation*}
    \sup_{\lambda\in K}\prob\left(\big|Y_t^\lambda-Z_t\big|>\frac{|z|}{2}\right)\leq\prob\Big(\|Z\|_{\cN_\kappa([0,t])}>C|z|-1\Big)
  \end{equation*}
  and the right-hand side is finally bounded by \cref{lem:norm_fou_tail}. This proves the corollary in the locally uniform case; if $b$ satisfies \hyperlink{psk}{$\mathbf{(PS)}_{0,H}$} the argument is similar.
\end{proof}

\subsection{Fractional Integrals and Derivatives}

Let $\alpha\in(0,1)$ and $T>0$. The \emph{(right-sided) Riemann-Liouville fractional integral} of order $\alpha$ is the operator $\cI_{+}^\alpha:L^1 \big([0,T],\R^n\big)\to L^1\big([0,T],\R^n\big)$ defined by
\begin{equation}\label{eq:fractional_integral}
  \cI_+^\alpha f(t)\define\frac{1}{\Gamma(\alpha)}\int_0^t (t-s)^{\alpha-1}f(s)\,ds,\qquad t\in[0,T],
\end{equation}
where $\Gamma(\alpha)=\int_0^{\infty} t^{\alpha-1} e^{-t}\,dt$ is Euler's gamma function. The name `fractional integral' is motivated by the mapping properties of the above operators on the spaces of H\"older continuous functions:

\begin{lemma}[{\cite[\para 3.1, Corollary 1]{Samko1993}}]\label{lem:holder_fractional_integral}
  Let $\alpha\in(0,1)$, $\beta\in[0,1]$, and $T>0$. Then, if $\alpha+\beta\neq 1$, there is a constant $C_{\alpha,\beta}>0$ such that
  \begin{equation*}
    \big\|\cI_+^{\alpha} f\big\|_{\C^{\alpha+\beta}}\leq C_{\alpha,\beta}\vertiii{f}_{\C^\beta}\qquad\forall\,f\in\C^\beta\big([0,T],\R^n\big).
  \end{equation*}
\end{lemma}

The next result is a simple special case of a result of Hardy and Littlewood \cite{Hardy1928}, see also \cite[\para 3, Theorem 3.5]{Samko1993}.  

\begin{lemma}\label{lem:hardy_littlewood}
  Let $\alpha\in(0,1)$. The operator $\cI_{+}^\alpha: L^2\big([0,T],\R^n\big)\to L^2\big([0,T],\R^n\big)$ is bounded.
\end{lemma}
\begin{proof}
  Straight-forward application of H\"older's inequality.
\end{proof}

It turns out that---on a suitable restriction of their domain---the fractional integral is invertible. Its inverse is given by the \emph{(right-sided) Riemann-Liouville fractional derivative} of order $\alpha\in(0,1)$, which is defined by
\begin{equation}\label{eq:fractional_derivative}
  \cI_+^{-\alpha} f(t)\define\frac{d}{dt}\cI_+^{1-\alpha} f(t)=\frac{1}{\Gamma(1-\alpha)}\frac{d}{dt}\int_0^t(t-s)^{-\alpha} f(s)\,ds,\qquad t\in[0,T].
\end{equation}
This operator is of course only well-defined if $\cI_+^{1-\alpha} f$ is absolutely continuous. In the sequel we shall only apply the fractional derivatives to functions $f\in\C^{\alpha+}\big([0,T],\R^n\big)$, for which $\cI_+^{1-\alpha} f$ is even continuously differentiable by virtue of \cref{lem:holder_fractional_integral}.

\begin{lemma}\label{lem:prelim_fractional_integral_inverse}
  Let $\alpha\in(0,1)$. Then 
  \begin{equation}\label{eq:prelim_derivative_left_inverse}
    \big(\cI_{+}^{-\alpha}\circ\cI_{+}^\alpha\big) f=f \qquad\forall\,f\in L^1\big([0,T],\R^n\big),
  \end{equation}
  as well as
  \begin{equation*}
    \big(\cI_{+}^\alpha\circ\cI_{+}^{-\alpha}\big) f=f \qquad\forall\,f\in\C^{\alpha+}\big([0,T],\R^n\big).
  \end{equation*}
\end{lemma}

\begin{proof}
  \Cref{eq:prelim_derivative_left_inverse} is \cite[Theorem 2.4]{Samko1993}. This result also states that, if $g\define\cI_{+}^{1-\alpha} f$ is absolutely continuous, then
  \begin{equation*}
    \big(\cI_{+}^\alpha\circ\cI_{+}^{-\alpha}\big) f(t)=f(t)-\frac{g(0)}{\Gamma(\alpha)}t^{\alpha-1},\qquad t\in[0,T].
  \end{equation*}
  If $f\in\C^\beta\big([0,T],\R^n\big)$ for some $\beta\in(\alpha,1]$, then $g\in\C^{1-\alpha+\beta}\big([0,T],\R^n\big)$ by \cref{lem:holder_fractional_integral}. Since $1-\alpha+\beta>1$, $g$ is absolutely continuous. One also easily checks that $g(0)=0$. This completes the proof.
\end{proof}

\begin{lemma}\label{lem:holder_fractional_derivative}
  Let $\alpha\in(0,1)$ and $\beta\in[\alpha,1]$. Then there is a constant $C_{\alpha,\beta}>0$ such that, for each $f\in\C^\beta\big([0,T],\R^n\big)$, 
  \begin{equation*}
    \big\|\cI_+^{-\alpha}f\big\|_{\C^{\beta-\alpha}}\leq C_{\alpha,\beta}\vertiii{f}_{\C^\beta}.
  \end{equation*}
\end{lemma}
\begin{proof}
  Notice that $\big\|\cI_+^{-\alpha}f\big\|_{\C^{\beta-\alpha}}\leq\big\|\cI^{1-\alpha}_+ f\big\|_{\C^{1+\beta-\alpha}}\leq C_{1-\alpha,\beta}\vertiii{f}_{\C^{\beta}}$ by \cref{lem:holder_fractional_integral}.
\end{proof}

\section{The Girsanov Theorem for the Liouville Process}\label{sec:Girsanov1}

The objective of this section is twofold: First, we apply the Girsanov theorem to the process \eqref{eq:innovation_sde}, see \cref{prop:density_girsanov}. This result prepares the representation of the density in terms of the Wiener-Liouville bridge, which is completed in \cref{prop:repres_psi}. In \cref{sec:reg_by_noise} we show that \eqref{eq:sde} is strongly well posed under the condition \ref{cond:h} as well as that \eqref{eq:innovation_sde} is weakly well posed.

\subsection{Statement and Application to the Conditional Density}

Let $(\F_t)_{t\geq 0}$ be the right-continuous completion of the natural filtration of the Wiener process driving the Liouville process $\tilde{B}$. Notice that there is a normalization constant $\varrho_H>0$ such that
\begin{equation*}
  \tilde{B}_t=\varrho_H\cI_+^{H-\frac12}W(t),\qquad t\geq 0,
\end{equation*}
for any $H\in(0,1)$, where we recall that $\cI_+^{\frac12-H}$ is defined by \eqref{eq:fractional_integral} and \eqref{eq:fractional_derivative}, respectively. It is convenient to introduce the following class of drifts:

\begin{definition}\label{def:fractional_admissible}
  Let $H\in(0,1)$ and $\Upsilon:\Omega\times\R_+\to\R^n$ be an $(\F_t)_{t\geq 0}$-adapted random process with continuous sample paths and $\Upsilon_0=0$ a.s. We say that $\Upsilon$ is an \emph{admissible fractional drift} up to time $T>0$ if there is a $\rho>0$ such that
  \begin{equation}\label{eq:fractional_admissible}
    \sup_{t\in[0,T]}\Expec{\exp\left(\rho \left|\frac{d}{dt}\Big(\cI_+^{\frac12-H}\Upsilon\Big)(t)\right|^2\right)}<\infty.
  \end{equation}  
\end{definition}

\begin{lemma}\label{lem:girsanov}
  Let $H\in(0,1)$ and suppose that $\Upsilon:\Omega\times\R_+\to\R^n$ is an admissible fractional drift up to time $T>0$. Then there is a probability measure $\Q_T\approx\prob$ under which $\big(\tilde{B}_t+\Upsilon_t\big)_{t\in[0,T]}$ is a Liouville process. The probability measure $\Q_T$ can be defined by the density
  \begin{equation*}
    \frac{d\Q_T}{d\prob}=\exp\left(-\frac{1}{\varrho_H}\int_0^T\Braket{\frac{d}{dt}\Big(\cI_+^{\frac12-H}\Upsilon\Big)(t),dW_t}-\frac{1}{2\varrho_H^2}\int_0^T\Big|\frac{d}{dt}\Big(\cI_+^{\frac12-H}\Upsilon\Big)(t)\Big|^2\,dt\right),
  \end{equation*}
  where $\varrho_H>0$ is a normalization constant.
\end{lemma}
\begin{proof}
  This is an immediate consequence of the Girsanov theorem for the standard Wiener process and the invertibility of the transformation $\cI_+^{H-\frac12}$, see \cref{lem:prelim_fractional_integral_inverse}. In fact, the former is applicable to $\big(W_t+\int_0^t v_s\,ds\big)_{t\in[0,T]}$ as soon as there is a $\rho>0$ such that
  \begin{equation*}
    \sup_{t\in[0,T]}\Expec{e^{\rho |v_t|^2}}<\infty,
  \end{equation*}
  see \cite[Chapter 7, Theorem 1.1]{Friedman1975}. Applying this with $v_t=\frac{d}{dt}\big(\cI_+^{\frac12-H}\Upsilon\big)(t)$ concludes the proof.
\end{proof}

The first technical step in the proof of the Gaussian-type lower bound is to establish fractional admissibility for the drift in \eqref{eq:innovation_sde}. More specifically, the following result on the solution $\Phi(\ell)$ of \eqref{eq:innovation_sde} holds:

\begin{lemma}\label{lem:novikov_satisfied}
  Let $H\in(0,1)$ and assume \ref{cond:h}. Then, for each $\ell\in\Cloc^{H-}(\R_+,\R^n)$, the process $\sigma^{-1}\int_0^\cdot b\big(\Phi_t(\ell)\big)\,dt$ is an admissible fractional drift up to any time $T>0$.
\end{lemma} 
\begin{proof}
  Let us first observe that, for each $\alpha\in(-1,1)$,
  \begin{equation*}
    \frac{d}{dt}\cI_+^\alpha\left(\int_0^\cdot b\big(\Phi_s(\ell)\big)\,ds\right)(t)=\Big(\cI_+^\alpha b\big(\Phi(\ell)\big)\Big)(t)\qquad\forall\,t\geq 0.
  \end{equation*}
  Let us fix $T>0$ and distinguish small and large Hurst parameters in the sequel:

  If $H<\frac12$, then 
  \begin{equation*}
    \sup_{t\in[0,T]}\left|\sigma^{-1}\Big(\cI_+^{\frac12-H}b\big(\Phi(\ell)\big)\Big)(t)\right|\lesssim T^{\frac12-H}\big\|b\big(\Phi(\ell)\big)\big\|_\infty\lesssim T^{\frac12-H}\Big(1+\big\|\Phi(\ell)\big\|_\infty\Big)
  \end{equation*}
  by (at most) linear growth of $b$. A straight-forward Gr\"onwall estimate shows that
  \begin{equation}\label{eq:gronwall_novikov}
    \big\|\Phi(\ell)\big\|_{\infty}\lesssim \Big(1+\|\ell\|_\infty+\|\tilde{B}\|_\infty\Big)e^{C T}.
  \end{equation}
  Since $\|\tilde{B}\|_\infty$ has Gaussian tails, we can choose $\rho>0$ sufficiently small such that \eqref{eq:fractional_admissible} holds true.
  
  For large Hurst parameters $H>\frac12$ we let $\alpha>1-\frac{1}{2H}$ be such that $b\in\C^\alpha(\R^n,\R^n)$ and use \cref{lem:holder_fractional_derivative} to find
  \begin{equation*}
    \sup_{t\in[0,T]}\left|\sigma^{-1}\Big(\cI_+^{\frac12-H}b\big(\Phi(\ell)\big)\Big)(t)\right|\lesssim T^{\gamma}\big\|b\big(\Phi(\ell)\big)\big\|_{\C^{\gamma}}\lesssim T^\gamma\big\|\Phi(\ell)\big\|_{\C^{\frac{\gamma}{\alpha}}}^\alpha
  \end{equation*}
  for each $\gamma\in\big(H-\frac12,\alpha H\big)$. Since $b$ grows at most linearly, it is easy to see that
   \begin{equation*}
     \big\|\Phi(\ell)\big\|_{\C^{\frac{\gamma}{\alpha}}}\lesssim \Big(\|\ell\|_{\C^{\frac{\gamma}{\alpha}}}+\|\tilde{B}\|_{\C^{\frac{\gamma}{\alpha}}} + \big\|\Phi(\ell)\big\|_\infty\Big).
   \end{equation*}
   Since $\|\tilde{B}\|_{\C^{\frac{\gamma}{\alpha}}}$ has Gaussian tails by \cref{prop:fernique}, the Gr\"onwall estimate \eqref{eq:gronwall_novikov} and an application of Cauchy-Schwarz imply \eqref{eq:fractional_admissible} for sufficiently small $\rho>0$. This completes the proof.
\end{proof}

The previous two lemmas imply the following result:
\begin{proposition}\label{prop:density_girsanov}
  For $\ell\in\Cloc^{H-}(\R_+,\R^n)$ let $\Phi(\ell)$ be a weak solution of \eqref{eq:innovation_sde}. Then there is a normalization constant $\varrho_H>0$ such that, for each $T>0$, the density of $\Phi_T(\ell)$ admits the expression 
  \begin{equation}\label{eq:density_girsanov}
    p_T(\ell;y)=\frac{1}{\big(\sqrt{2\pi}\varrho_H T^H\big)^n\det(\sigma)}\exp\left(-\frac{\big|\sigma^{-1}\big(y-\ell(T)\big)\big|^2}{\varrho_H^2 T^{2H}}\right)\Psi_T(\ell;y),
  \end{equation}
  where
  \begin{align}
    \Psi_T(\ell;y)&\define\mathbb{E}\left[\exp\left(\int_0^T \big\langle\fK_t^{\ell}, dW_t\big\rangle-\frac{1}{2}\int_0^T \big|\fK_t^{\ell}\big|^2 dt\right)\,\middle|\,\sigma\tilde{B}_T+\ell(T)=y\right], \label{eq:psi_girsanov}\\
    \fK_t^{\ell}&\define(\varrho_H\sigma)^{-1}\Big(\cI_{+}^{\frac12-H}b\big(\ell+\sigma\tilde{B}\big)\Big)(t). \nonumber
  \end{align}
\end{proposition}
\begin{proof}
  \Cref{lem:girsanov,lem:novikov_satisfied} imply that, for any $f:\R^n\to\R$ bounded Borel measurable,
  \begin{align*}
    &\phantom{=}\int_{\R^n} f(y)p_T(\ell;y)\,dy=\mathbb{E}_\prob\left[f\left(\ell(T)+\int_0^T b\big(\Phi_t(\ell)\big)\,dt+\sigma\tilde{B}_T\right)\right] \\
    &=\mathbb{E}_{\Q_T^{\ell}} \bigg[f\left(\ell(T)+\varrho_H\sigma \Big(\cI_{+}^{H-\frac12}W^{\mathbb{Q}_T^\ell}\Big)(T)\right) \\
    &\phantom{=\mathbb{E}_{\Q_T^{z,w}} \bigg[}\exp\left(\int_0^T \Braket{\widehat{\fK}_t^{\ell}, dW_t^{\mathbb{Q}_T^\ell}}-\frac{1}{2}\int_0^T \big|\widehat{\fK}_t^{\ell}\big|^2 dt\right)\bigg],
  \end{align*}
  where $\big(W^{\mathbb{Q}_T^\ell}_t\big)_{t\in[0,T]}$ is a standard Wiener process under $\Q_T^{\ell}$ and 
  \begin{equation*}
    \widehat{\fK}^{\ell}_t=(\varrho_H\sigma)^{-1}\Big(\cI_{+}^{\frac12-H}b\Big(\ell+\sigma\cI_{+}^{H-\frac12}W^{\mathbb{Q}_T^\ell}\Big)\Big)(t).
  \end{equation*}
  Note that the drifted $\Q^{\ell}_T$-Liouville process $B^{\mathbb{Q}_T^\ell}=\ell + \varrho_H\sigma\cI_+^{H-\frac12}W^{\mathbb{Q}_T^\ell}$ has distribution $B^{\mathbb{Q}_T^\ell}_T\sim N\left(\ell(T), \varrho_H^2 T^{2H}\sigma\sigma^\top\right)$. Let us abbreviate its density by
  \begin{equation*}
    \varphi_T^{\ell}(y)\define\frac{1}{(2\pi)^{\frac{n}{2}}(\varrho_H T^H)^n\det(\sigma)}\exp\left(-\frac{\big|\sigma^{-1}\big(y-\ell(T)\big)\big|^2}{\varrho_H^2 T^{2H}}\right).
  \end{equation*}
  We see that
  \begin{align*}
    &\phantom{=}\int_{\R^n} f(y)p_T(\ell;y)\,dy \\
    &=\int_{\R^n}f(y)\varphi_T^{\ell}(y)\;\mathbb{E}_{\Q_T^{\ell}} \left[\exp\left(\int_0^T \Braket{\widehat{\fK}_t^{\ell}, dW_t^{\mathbb{Q}_T^\ell}}-\frac{1}{2}\int_0^T \big|\widehat{\fK}_t^{\ell}\big|^2 dt\right)\,\middle|\,B_T^{\mathbb{Q}_T^\ell}=y\right]\,dy,
  \end{align*}
  where the expectation is, as usual, understood in the sense of regular conditional probabilities.
\end{proof}

\subsection{Regularization by Noise}\label{sec:reg_by_noise}

In this section we show that under the very mild regularity conditions on the drift vector field imposed in \cref{thm:lower_bound_tindel} the equations \eqref{eq:sde} and \eqref{eq:innovation_sde} are well posed. We begin with \eqref{eq:sde}. This is essentially a well known result: Indeed, for \emph{bounded} drift vector fields this was proven in \cite[Theorem 1.9]{Catellier2016}, see also \cite{Nualart2002} for the simpler one-dimensional case as well as the review article \cite{Galeati2021}. The well-posedness under a linear growth condition can be deduced by a standard stopping time argument, which we include for the reader's convenience below:
\begin{proposition}\label{prop:strong_well_posedness_tindel}
  Let $\sigma\in\Lin{n}$ be invertible and $y_0\in\R^n$ be deterministic. Suppose that $b:\R^n\to\R^n$ satisfies assumption \ref{cond:h}. Then the equation \eqref{eq:sde} has a pathwise unique strong solution.
\end{proposition}
\begin{proof}
  Let $R>0$ and $\psi_R:\R^n\to\R_+$ be a smooth cutoff function which is equal to $1$ on $\{|x|\leq R\}$ and vanishes outside of $B_{2R}$. As explained above, $b_R(y)\define\psi_R(y) b(y)$, being bounded H\"older continuous (resp. measurable), falls in the regime of \cite[Theorem 1.9]{Catellier2016} showing the existence of a strong solution to \eqref{eq:sde} until the random time $\tau_R\define\inf\{t\geq 0:\,|Y_t|>R\big\}$. By a Gr\"onwall argument we observe that
  \begin{equation}\label{eq:gronwall_strong}
    |Y_t|\leq \big(CT+|y_0|+\|B\|_{\infty}\big) e^{C T}\qquad\forall\, t\in[0,T]
  \end{equation}
  Hence, $\tau_R\wedge T\to T$ almost surely as $R\to\infty$. Since the terminal time $T>0$ was arbitrary the strong existence follows.

  The pathwise uniqueness follows from the local pathwise uniqueness.
\end{proof}

As a second step we shall show the well-posedness of the conditional evolution \eqref{eq:innovation_sde} under the regularity condition \ref{cond:h}. To this end, we recall that a weak solution to \eqref{eq:innovation_sde} consists of a filtered probability space and an adapted process $\big(\Phi(\ell), \tilde{B}\big)$ such that $\tilde{B}$ is a Liouville process with the given Hurst parameter and $\Phi(\ell)$ satisfies \eqref{eq:innovation_sde}.

\begin{lemma}\label{lem:innovation_sde_well_posed}
  Let $\sigma\in\Lin{n}$ be invertible and suppose that $b:\R^n\to\R^n$ satisfies assumption \ref{cond:h}. Then, for each $\ell\in\Cloc^{H-}(\R^n,\R^n)$, the equation \eqref{eq:innovation_sde} has a weak solution which is unique in law.
\end{lemma}
\begin{proof}
  To show the existence of a weak solution to \eqref{eq:innovation_sde} we first notice that 
  \begin{equation*}
    \Upsilon_t\define - \sigma^{-1}\int_0^t b\big(\ell(s)+\sigma\tilde{B}_s\big)\,ds, \qquad t\in[0,T],
  \end{equation*}
  is an admissible fractional drift in the sense of \cref{def:fractional_admissible} up to any time $T>0$. Indeed, this can be shown by the same reasoning as in \cref{lem:novikov_satisfied}. It follows from \cref{lem:girsanov} that
  \begin{equation*}
    \tilde{B}^{\mathbb{Q}_T^\ell}_t\define\tilde{B}_t-\sigma^{-1}\int_0^t b\big(\ell(s)+\sigma\tilde{B}_s\big)\,ds,\qquad t\in[0,T],
  \end{equation*} 
  is a Liouville process under $\mathbb{Q}_T^\ell$ adapted to the natural filtration $(\F_t^{\tilde{B}})_{t\in[0,T]}$ of $\tilde{B}$. Hence, declaring $\Phi(\ell)\define \ell+\sigma\tilde{B}$ the tuple $\big(\Phi(\ell),\tilde{B}^{\mathbb{Q}_T^\ell},\mathbb{Q}_T^\ell,(\F_t^{\tilde{B}})_{t\in[0,T]}\big)$ is a weak solution to \eqref{eq:innovation_sde}. 

  To see that uniqueness in law holds, we let $A\subset\C\big([0,T],\R^n\big)$ be a Borel set and observe that by \cref{lem:girsanov,lem:novikov_satisfied}
  \begin{align*}
    \prob\big(\Phi(\ell)\in A\big)&=\mathbb{E}\bigg[\1_{\big\{\ell+\sigma\tilde{B}\in A\big\}}\exp\bigg(\frac{1}{\varrho_H}\int_0^T\Braket{\sigma^{-1}\cI_+^{\frac12-H} b\big(\ell+\sigma\tilde{B}\big)(t),dW_t} \\
    &\phantom{=\mathbb{E}\bigg[}- \frac{1}{2\varrho_H^2}\int_0^T\Big|\sigma^{-1}\cI_+^{\frac12-H} b\big(\ell+\sigma\tilde{B}\big)(t)\Big|^2\,dt\bigg)\bigg],
  \end{align*}
  where $W$ is a standard Wiener process and $\tilde{B}$ is the Riemann-Liouville process induced by $W$. Since this identity holds for every solution to \eqref{eq:innovation_sde}, uniqueness in law follows.
\end{proof}

While we believe that it is actually possible to show pathwise uniqueness for \eqref{eq:innovation_sde} by the non-linear Young integration developed in \cite{Catellier2016}, the `weak' well-posedness statement of \cref{lem:innovation_sde_well_posed} suffices for our purposes.

\section{The Wiener-Liouville Bridge}\label{sec:Girsanov2}

The representation of the density given in \cref{prop:density_girsanov} is based on a conditional expectation with respect to the terminal value $y$. In view of our objectives, we have to work a bit harder on this conditioning. More specifically, we need to introduce Wiener-Liouville bridges from $0$ to $\x\in\R^n$ in order to produce a more tractable representation of the density. Such Volterra, or more general Gaussian, bridges have been previously studied in the literature, see \cite{Baudoin2007,Gasbarra2007} and references therein. 

\subsection{Basic Properties}

Let 
\begin{equation*}
  \tilde{B}_t=\varrho_H\int_0^t (t-s)^{H-\frac12}\,dW_s
\end{equation*}
be a Liouville process and let $\x\in\R^n$. The regular conditional law $\L\big((W_t)_{t\in[0,T]}\,|\,\tilde{B}_T=\x\big)$ is called the \emph{Wiener-Liouville bridge} from $0$ to $\x$. We shall make use of the following lemma:

\begin{lemma}\label{lem:wiener_conditioned_sde}
  Any weak solution of the path-dependent SDE
  \begin{equation}\label{eq:conditioning_pathdependent_sde}
    dX_t^\x=\frac{2H}{\varrho_H}(T-t)^{-H-\frac12}\left(\x-\varrho_H\int_0^t (T-s)^{H-\frac12}\,dX_{s}^\x\right)\,dt+dW_t,\qquad X_0^\x=0,
  \end{equation}
  satisfies $(X_t^\x)_{t\in[0,T]}\overset{d}{=}\L\big((W_t)_{t\in[0,T]}\,|\,\tilde{B}_T=\x\big)$ for each $\x\in\R^n$.
\end{lemma}
\begin{proof}
  We aim to find a transformation sending $\tilde{B}_T$ to the endpoint of a standard Wiener process $\tilde{W}$ and to subsequently also express $(W_t)_{t\in[0,T]}$ in terms of $\tilde{W}$. To this end, we declare
  \begin{equation*}
    U_t\define\varrho_H\int_0^{t} (T-s)^{H-\frac12}\,dW_s,\qquad t\in[0,T).
  \end{equation*}
  Notice that $U$ is an $L^2$-bounded martingale with limit $\lim_{t\uparrow T} U_t=\tilde{B}_{T}$ and quadratic variation
  \begin{equation*}
    \varsigma(t)\define\braket{U}_t=\varrho_H^2\int_0^t (T-s)^{2H-1}\,ds=\frac{\varrho_H^2}{2H}\Big(T^{2H}-(T-t)^{2H}\Big),\qquad t\in[0,T].
  \end{equation*}
  Set
  \begin{equation*}
    \tilde{W}_t\define U_{\varsigma^{-1}(t)}=\varrho_H\int_0^{\varsigma^{-1}(t)}(T-s)^{H-\frac12}\,dW_s,\qquad t\in[0,\varsigma(T)].
  \end{equation*}
  Notice that, for $t\in[0,T)$,
  \begin{equation*}
    W_t=\frac{1}{\varrho_H}\int_0^t (T-s)^{\frac12-H}\,dU_s,
  \end{equation*}
  which has a limit as $t\uparrow T$. For any continuous semimartingale $(Z_t)_{t\in[0,T]}$ the mapping
  \begin{equation*}
    \phi(Z)_t\define\frac{1}{\varrho_H}\int_0^t (T-s)^{\frac12-H}\,dZ_s,\qquad t\in[0,T],
  \end{equation*}
  has inverse $\phi^{-1}(\bar Z)_t=\varrho_H\int_0^t (T-s)^{H-\frac12}\,d\bar Z_s$ (provided these integrals are well defined). Since $(W, U)=\big(\phi(\tilde{W}\circ\varsigma), \tilde{W}\circ\varsigma\big)$, the conditional distribution of $(W_t)_{t\in[0,T]}$ given $U_T=\x$ equals $\phi\big(Z^\x\circ\varsigma\big)$, where $(Z^\x_t)_{t\in[0,\varsigma(T)]}$ is the standard Wiener bridge from $0$ to $\x$ satisfying
  \begin{equation*}
    dZ_t^\x=\frac{\x-Z_t^\x}{\varsigma(T)-t}\,dt+d\tilde{W}_t,\qquad t\in\big[0,\varsigma(T)\big).
  \end{equation*}
  Observe that
  \begin{equation*}
    dZ_{\varsigma(t)}^\x=2H\frac{\x-Z_{\varsigma(t)}^\x}{T-t}\,dt+\varrho_H (T-t)^{H-\frac12}\,dW_t,\qquad t\in[0,T).
  \end{equation*}
  For $X^\x\define\phi\big(Z^\x\circ\varsigma\big)$ we consequently have
  \begin{equation*}
    X_t^\x=\frac{1}{\varrho_H}\int_0^t(T-s)^{\frac12-H}\,dZ_{\varsigma(s)}^\x=\frac{2H}{\varrho_H}\int_0^t (T-s)^{-H-\frac12}\big(\x-Z_{\varsigma(s)}^\x\big)\,ds+W_t.
  \end{equation*}
  Insert $Z_{\varsigma(t)}^\x=\phi^{-1}(X^\x)_t=\varrho_H\int_0^t (T-s)^{H-\frac12}\,dX_s^\x$ to conclude.
\end{proof}

\begin{corollary}\label{cor:condition_wiener}
  For each $\x\in\R^n$, the semimartingale 
  \begin{equation}\label{eq:wiener_conditioned}
    dX_t^\x=\frac{2H}{\varrho_H}(T-t)^{H-\frac12}\left(\frac{\x}{T^{2H}}-\varrho_H\int_0^t(T-s)^{-H-\frac12}\,dW_s\right)\,dt+dW_t,\quad X_0^\x=0,
  \end{equation}
  satisfies $(X_t^\x)_{t\in[0,T]}\overset{d}{=}\L\big((W_t)_{t\in[0,T]}\,|\,\tilde{B}_T=\x\big)$.
\end{corollary}

\begin{proof}
  We derive \eqref{eq:wiener_conditioned} by solving the path-dependent SDE of \cref{lem:wiener_conditioned_sde}. To this end, it is natural to make the ansatz
  \begin{equation}\label{eq:ansatz_x}
    X_t^\x=\frac{1}{\varrho_H}\int_0^t\frac{d L_s^\x}{(T-s)^{H-\frac12}}.
  \end{equation}
  Inserting this into \eqref{eq:conditioning_pathdependent_sde}, we see that $L^\x$ solves
  \begin{equation*}
    dL_t^\x=2H(T-t)^{2H-1}\left(\frac{\x}{T^{2H}}-L_t^\x\right)\,dt+\varrho_H(T-t)^{H-\frac12}\,dW_t
  \end{equation*}
  and it follows from variation of constants that
  \begin{equation}\label{eq:semimartingale_l}
    dL_t^\x=2H(T-t)^{2H-1}\left(\frac{\x}{T^{2H}}-\varrho_H\int_0^t(T-s)^{-H-\frac12}\,dW_s\right)\,dt+\varrho_H (T-t)^{H-\frac12}\,dW_t,\quad L_0^\x=0.
  \end{equation}
  Returning to \eqref{eq:ansatz_x}, we find the semimartingale decomposition \eqref{eq:wiener_conditioned} and it is straight-forward to check that this process actually solves the equation \eqref{eq:conditioning_pathdependent_sde}.
\end{proof}

As a next step, we study the process \eqref{eq:wiener_conditioned} in further detail:
\begin{lemma}\label{lem:conditioning_holder}
  For each $\gamma<\frac12$ and each $T>0$, there are $C>0$ and $\rho_0>0$ such that, for each $\x\in\R^n$, we have
  \begin{equation*}
    \sup_{\rho\leq\rho_0}\Expec{\exp\left(\rho\|X^\x\|_{\C^\gamma([0,T])}^2\right)}\leq e^{C(1+|\x|^2)},
  \end{equation*}
  where $X^\x$ is defined in \eqref{eq:wiener_conditioned}.
\end{lemma}
\begin{proof}
  It is enough to show that the Gaussian process $\mathfrak{G}_t\define(T-t)^{H-\frac12+1-\gamma}\int_0^t(T-s)^{-H-\frac12}\,dW_s$, $t\in[0,T]$, satisfies $\Expec{e^{\rho\|\mathfrak{G}\|_{\infty}^2}}<\infty$. This is however a simple consequence of \cref{prop:fernique} and the fact that $\sup_{t\in[0,T]}\Expec{|\mathfrak{G}_t|^2}<\infty$.
\end{proof}

\subsection{Wiener-Liouville Bridge Representation of the Conditional Density}

With the help of $X^\x$, we can give a novel representation of the density $p_T(\ell;\cdot)$ of the solution to \eqref{eq:innovation_sde}. To this end, we need the following technical lemma.

\begin{lemma}\label{lem:l2_bound_l}
  Assume that $b:\R^n\to\R^n$ satisfies \ref{cond:h}. Let $T>0$ and let $\alpha>1-\frac{1}{2H}$ be such that $b\in\C^\alpha(\R^n,\R^n)$ if $H>\frac12$ and $\alpha=1$ otherwise. Then there is a constant $C>0$ such that, for each $(\ell,\x)\in\Cloc^{H-}(\R_,\R^n)\times\R^n$, the process
  \begin{equation}\label{eq:definition_l}
    \mathfrak{L}^{\ell,\x}_t\define(\varrho_H\sigma)^{-1}\left(\cI_+^{\frac12-H} b\Big(\ell+\varrho_H\sigma\cI_+^{H-\frac12}X^\x\Big)\right)(t),\qquad t\in[0,T],
  \end{equation}
  where $X^\x$ is the semimartingale defined in \eqref{eq:wiener_conditioned}, satisfies the estimate
  \begin{equation}\label{eq:sup_norm_l}
    \big\|\fL^{\ell,\x}\big\|_\infty\lesssim \begin{cases}
    T^{\frac12-H}\Big(1+|\x|+\vertiii{\ell}_{\C^\gamma}+T^\gamma\|X^\x\|_{\C^{\gamma-H+\frac12}}\Big), & H<\frac12, \\
    T^{\gamma}\Big(1+|\x|+\vertiii{\ell}_{\C^{\frac{\gamma}{\alpha}}}+\|X^\x\|_{\C^{\frac{\gamma}{\alpha}-H+\frac12}}\Big), & H>\frac12,
    \end{cases}
  \end{equation}
  for any $\gamma\in\big(0\vee(H-\frac12),\alpha H\big)$.
  In particular, there is a $\bar{\gamma}<H$ such that, for each $T>0$,
  \begin{equation}\label{eq:l2_bound_l}
   \Expec{\int_0^T\big|\mathfrak{L}^{\ell,\x}_t\big|^2\,dt}^{\frac12}\lesssim 1+|\x|+\vertiii{\ell}_{\C^{\bar{\gamma}}}
  \end{equation}
  for each $(\ell,\x)\in\Cloc^{H-}(\R_+,\R^n)\times\R^n$.
\end{lemma}

\begin{proof}
  We split the cases of small and large Hurst parameters. If $H<\frac12$, then the estimate $\big\|\cI_+^{\frac12-H} f\big\|_\infty\lesssim T^{\frac12-H}\|f\|_\infty$ and the most linear growth of $b$ imply 
  \begin{equation*}
    \big\|\fL^{\ell,\x}\big\|_\infty\lesssim T^{\frac12-H}\Big(1+\|\ell\|_\infty+\big\|\cI_+^{H-\frac12}X^\x\big\|_\infty\Big).
  \end{equation*}
  By \cref{lem:holder_fractional_derivative} the latter term can bounded by the $\C^{\gamma}$-norm of $X^\x$ for any $\gamma\in(0,H)$ which gives \eqref{eq:sup_norm_l}.
 
  For $H>\frac12$ fix $\gamma\in\big(H-\frac12,\alpha H\big)$ and use \cref{lem:holder_fractional_derivative} to estimate
  \begin{equation}\label{eq:holdest2}
    \big\|\fL^{\ell,\x}\big\|_\infty\lesssim T^{\gamma}\|b\|_{\C^\alpha}\left(\|\ell\|_{\C^{\frac{\gamma}{\alpha}}}^\alpha+\Big\|\cI_+^{H-\frac12}X^\x\Big\|_{\C^{\frac{\gamma}{\alpha}}}^\alpha\right)\lesssim T^{\gamma}\Big(1+\|\ell\|_{\C^{\frac{\gamma}{\alpha}}}+\|X^\x\|_{\C^{\frac{\gamma}{\alpha}-H+\frac12}}\Big),
  \end{equation}
  where the last step follows from \cref{lem:holder_fractional_integral}.
  
  Finally, the $L^2$-bound immediately follows \cref{lem:conditioning_holder}.
\end{proof}

\begin{proposition}\label{prop:repres_psi} 
Assume that $b:\R^n\to\R^n$ satisfies \ref{cond:h}. Let $(\ell,y)\in\Cloc^{H-}(\R_+,\R^n)\times\R^n$ and set $\x\define \sigma^{-1}\big(y-\ell(T)\big)$. Let $\Psi_T$ and $\mathfrak{L}^{\ell,\x}$ be defined by \eqref{eq:psi_girsanov} and \eqref{eq:definition_l}, respectively. Then we have
\begin{equation}\label{eq:Psicontinue} 
  \Psi_T(\ell;y)=\Expec{\exp\left(\int_0^T\big\langle\mathfrak{L}^{\ell,\x}_t,dX_t^\x\big\rangle-\frac{1}{2}\int_0^T\big|\mathfrak{L}^{\ell,\x}_t\big|^2\,dt\right)}.
\end{equation}
\end{proposition}

\begin{proof}
  This is a direct consequence of \cref{prop:density_girsanov} and \cref{lem:wiener_conditioned_sde,lem:l2_bound_l}.
\end{proof}

\subsection{Gaussian Tail Estimates for the Conditional Density}

With the representation given in \cref{prop:repres_psi} we are now ready to establish Gaussian tail estimates of $p_T(\ell;\cdot)$, which will form a central part in the proofs of the main theorems.

\begin{lemma}\label{lem:exponential_bound_l}
  Assume \ref{cond:h} and $\ell\in\Cloc^{H-}\big(\R_+,\R^n\big)$. Let $\rho>0$. Then there are $\beta<H$, $T_0>0$ as well as $c>0$ depending only on $\rho$ and the drift $b$ such that
  \begin{equation*}
    \sup_{T\leq T_0}\Expec{\exp\left(\rho\int_0^T\big|\fL_t^{\ell,\x}\big|^2\,dt\right)}\lesssim e^{cT_0\big(1+|\x|^2+\vertiii{\ell}^2_{\C^\beta}\big)}.
  \end{equation*}
\end{lemma}
\begin{proof}
  This follows from \eqref{eq:sup_norm_l} and \cref{lem:conditioning_holder}.
\end{proof}

Let $(M_t)_{t\geq 0}$ be a local martingale. Recall that the \emph{Dol\'eans-Dade exponential} of $M$ is defined by
\begin{equation*}
  \E(M)_t\define\exp\left(M_t-\frac12\braket{M}_t\right),\qquad t\geq 0.
\end{equation*}
It is well known to satisfy the following moment bound
\begin{equation}\label{eq:ui_doleans}
  \Expec{\E(M)_T^p}\leq\sqrt{\Expec{\exp\big(p(2p-1)\braket{M}_T\big)}},
\end{equation}
which readily follows from an application of Cauchy-Schwarz to the identity
\begin{equation*}
  \E(M)_T^p=\sqrt{\E(2pM)_T} \sqrt{e^{p(2p-1)\braket{M}_T}}
\end{equation*}
and the fact that $\E(2pM)$ is a super-martingale.

\begin{proposition}\label{prop:lower_bound_conditioned_process}
  Assume that the drift satisfies \ref{cond:h}. Then there is a $\gamma<H$ such that the following hold: 
  \begin{enumerate}
    \item\label{it:conditioned_density_1} For any $T_0>0$, there exist constants $c,C>0$ such that, for each $\ell\in\Cloc^{H-}(\R_+,\R^n)$, the density of $\Phi_T(\ell)$ admits the lower bound 
      \begin{equation*}
        p_T(\ell;y)\geq \frac{C e^{-c\vertiii{\ell}^2_{\C^\gamma}}}{T^{nH}}\exp\left(-\frac{c\big|y-\ell(T)\big|^2}{T^{2H}}\right) \qquad\forall\,y\in\R^n,
      \end{equation*}
      for all $T\in(0,T_0]$.

    \item\label{it:conditioned_density_2} There exist $T_0>0$ and constants $c, C,\eta>0$ such that, for each $\ell\in\Cloc^{H-}\big(\R_+,\R^n\big)$, the density of $\Phi_T(\ell)$ admits the upper bound 
    \begin{equation}\label{eq:upper_bound_conditional_density}
      p_T(\ell;y)\leq \frac{e^{C(1+T^\eta\vertiii{\ell}_{\C^\gamma}^2)}}{T^{nH}}\exp\bigg(-\frac{c\big|y-\ell(T)\big|^2}{T^{2H}}\bigg) \qquad\forall\,y\in\R^n,
    \end{equation}
    for all $T\in(0, T_0]$.
  \end{enumerate}
\end{proposition}

\begin{remark}
  It is worth noting that part \ref{it:conditioned_density_1} of \cref{prop:lower_bound_conditioned_process} holds for \emph{any} positive time, whereas \ref{it:conditioned_density_2} only holds on a sufficiently small time interval. This has no consequences for our study of the stationary density since, in view of \cref{prop:integral_density}, we can choose \emph{any} $T_0>0$. However, it poses an additional challenge in the proof of \cref{thm:lower_bound_tindel}, which we overcome in \cref{sec:proof_tindel}.
\end{remark}

\begin{proof}[Proof of \cref{prop:lower_bound_conditioned_process}]
  \ref{it:conditioned_density_1} By \eqref{eq:density_girsanov} it is enough to lower-bound the expression $\Psi_T(\ell;y)$ given in \cref{prop:repres_psi}. We deduce from this representation and Jensen's inequality that 
 \begin{equation*}
    \Psi_T(\ell;y)\geq\exp\left(\Expec{\int_0^T\big\langle\mathfrak{L}^{\ell,\x}_t,dX_t^\x\big\rangle-\frac{1}{2}\int_0^{T}\big|\mathfrak{L}^{\ell,\x}_t\big|^2\,dt}\right),
  \end{equation*}
  where we recall that $\x=\sigma^{-1}\big(y-\ell(T)\big)$. Since $\int_0^\cdot \big\langle\mathfrak{L}^{z,w,\x}_t,dW_t\big\rangle$ is a martingale by virtue of \cref{lem:l2_bound_l}, abbreviating the finite variation part of $X^\x$ by
  \begin{equation}\label{eq:definition_k}
    K_t^\x\define\frac{2H}{\varrho_H}(T-t)^{H-\frac12}\left(\frac{\x}{T^{2H}}-\int_0^t (T-s)^{-H-\frac12}\,dW_s\right),\qquad t\in[0,T],
  \end{equation}
  see \cref{cor:condition_wiener}, we see that there is a constant $c>0$ such that the lower bound
  \begin{equation}\label{eq:lower_bound_psi_l}
     \Psi_T(\ell;y)\geq\exp\left(-c\,\Expec{\int_0^T\big|\mathfrak{L}^{\ell,\x}_t\big|\big|K_t^\x\big|\,dt+\frac12\int_0^T\big|\mathfrak{L}^{\ell,\x}_t\big|^2\,dt}\right)
   \end{equation}
   holds. Notice that there are $\eta\in(0,1)$, $\bar{\eta}\in(1,\infty)$, $\beta<\frac12$, and $\gamma<H$ such that
  \begin{align}
    \int_0^T\big|\mathfrak{L}^{\ell,\x}_t\big|\big|K^\x_t\big|\,dt&\lesssim \big\|\fL^{\ell,\x}\big\|_\infty\Big(T^{\frac12-H}|\x| + T\|\mathfrak{H}^T\|_\infty\Big) \nonumber\\
    &\lesssim \big(T^\eta+T^{\bar{\eta}}\big)\left(1+|\x|^{2}+\vertiii{\ell}_{\C^\gamma}+\|X^\x\|_{\C^\beta}\right)\big(1+\|\mathfrak{H}^T\|_\infty\big), \label{eq:bound_product_integral}
  \end{align}
  where $\mathfrak{H}_t^T\define (T-t)^{H-\frac12}\int_0^t (T-s)^{-H-\frac12}\,dW_s$ and the last step follows from \cref{lem:l2_bound_l}. By \cref{prop:fernique,lem:conditioning_holder} the norms $\|X^\x\|_{\C^\beta}$ and $\|\mathfrak{H}^T\|_\infty$ have Gaussian tails uniformly in $T\in(0,T_0]$, whence the expectation of the above integral is proportional to $1+|\x|^{2}+\vertiii{\ell}_{\C^\gamma}$. Combining this with the $L^2$-bound \eqref{eq:l2_bound_l} we further bound \eqref{eq:lower_bound_psi_l}:
  \begin{equation*}
    \Psi_T(\ell;y)\geq\exp\left(-c\left(|\x|^2+\vertiii{\ell}_{\C^\gamma}^2\right)\right)
  \end{equation*}
  for a potentially increased $c>0$. Inserting this in the expression from \cref{prop:density_girsanov}, we obtain
  \begin{equation*}
    p_T(\ell;y)\geq \frac{C e^{-c\vertiii{\ell}^2_{\C^\gamma}}}{T^{nH}}\exp\left(-\frac{c|\x|^2}{T^{2H}}-c|\x|^2\right)\geq \frac{C e^{-c\vertiii{\ell}^2_{\C^\gamma}}}{T^{nH}}\exp\left(-\frac{c(1+T_0^{2H})|\x|^2}{T^{2H}}\right)
  \end{equation*}
  for all $y\in\R^n$ and the proof of the first statement is complete.\\

  \ref{it:conditioned_density_2} It is enough to assume $T_0<1$. We upper-bound $\Psi_T(\ell;y)$ uniformly in $T\in(0,T_0]$: Recall the definition \eqref{eq:definition_k}. It follows from \cref{lem:l2_bound_l}, \eqref{eq:ui_doleans}, and the Cauchy-Schwarz inequality that
  \begin{equation}\label{eq:psi_upper_bound}
  \Psi_T(\ell;y)\leq \Expec{\exp\left(2\int_0^T\big|\mathfrak{L}^{\ell,\x}_t\big|\big|K^\x_t\big|\,dt\right)}^{\frac{1}{2}}\Expec{\exp\left(6\int_0^T\big|\mathfrak{L}^{\ell,\x}_t\big|^2\,dt\right)}^{\frac{1}{4}} 
  \end{equation}
  Since $T<1$ by assumption, the bound \eqref{eq:bound_product_integral} reduces to
  \begin{equation*}
    \int_0^T\big|\mathfrak{L}^{\ell,\x}_t\big|\big|K^\x_t\big|\,dt\lesssim T^\eta\left(1+|\x|+\vertiii{\ell}_\gamma+\|X^\x\|_{\C^\beta}\right)\big(1+\|\mathfrak{H}\|_\infty\big).
  \end{equation*}
  Hence, applying \cref{lem:conditioning_holder} and a Cauchy-Schwarz argument we find
  \begin{equation}\label{eq:upperbound_cauchy}
    \Expec{\exp\left(2\int_0^T\big|\mathfrak{L}^{\ell,\x}_t\big|\big|K^\x_t\big|\,dt\right)}\lesssim\exp\left(cT^\eta \left(|\x|^2+\vertiii{\ell}_{\C^\gamma}^2\right)\right)
  \end{equation}
  for some $c>0$, provided we choose $T_0>0$ sufficiently small. Inserting this back into \eqref{eq:psi_upper_bound}, the proof is completed by appealing to \cref{lem:exponential_bound_l,prop:repres_psi}. In fact, the term $e^{c T^\eta |\x|^2}\leq e^{2c T^\eta (|y|^2 + \vertiii{\ell}_{\C^\gamma}^2)}$ can be absorbed into the other factors in the bound \eqref{eq:upper_bound_conditional_density} (upon potentially decreasing $T_0>0$) and the proof is complete.
\end{proof}

Next we prove that the mapping $y\mapsto\Psi_T(\ell;y)$ is differentiable under appropriate regularity assumptions on the drift $b$. To this end, we first establish a differentiability result on the expectation of Dol\'eans-Dade exponentials of parameter-dependent semimartingales. For convenience we shall use a common abbreviation of the stochastic integral: $f\bigcdot W=\int_0^\cdot f_t\,dW_t$.

\begin{lemma}\label{lem:parameter_doleans_dade}
  Let $f:\R^m\times\R_+\times\Omega\to\R$ be continuous in the first two arguments and suppose that the following hold:
  \begin{itemize}
    \item $f(\cdot,t)\in\C^k(\R^m)$ for each $t\geq 0$ and, for each $\balpha\in\N_0^m$ with $|\balpha|\leq k$, the mapping $(\xi,t)\mapsto\partial_\xi f(\xi,t)$ is continuous and
    \begin{equation*}
      \int_0^T\Expec{\big|\partial_\xi f(\xi,t)\big|^2}\,dt<\infty
    \end{equation*}
    locally uniform in $\xi$.
    
    \item There is a $\rho>6$ such that $\Expec{\exp\left(\rho\int_0^T \big|f (\xi,t)\big|^2\,dt\right)}<\infty$ and $\big\|\big(\partial_\xi^{\balpha} f(\xi,\cdot)\bigcdot W\big)_T\big\|_{L^p}<\infty$ locally uniform in $\xi$ for each $p\geq 1$ and each $\balpha\in\N_0^m$ with $|\balpha|\leq k$.
    
    \item Let $Z:\R^m\times\Omega\to\R$ be a random variable for which $\xi\mapsto Z(\xi,\cdot)$ is in $\C^k(\R^m)$ almost surely and the derivatives are measurable. Assume also that $\Expec{\exp\big(\rho Z(\xi)\big)}<\infty$ for some $\rho>2$ and $\big\|\partial_\xi^{\balpha} Z(\xi)\big\|_{L^p}<\infty$ for each $p\geq 1$ and each $\balpha\in\N_0^m$ with $|\balpha|\leq k$, both locally uniform in $\xi$.
   \end{itemize}  
   Then the mapping $\xi\mapsto\Expec{\E\big(f(\xi,\cdot)\bigcdot W\big)_Te^{Z(\xi)}}$ is $\C^k\big(\R^m)$. 

   Moreover, there is a polynomial $\mathfrak{p}$ such that 
  \begin{align*}
    &\phantom{=}\partial_\xi^{\balpha}\Expec{\E\big(f(\xi,\cdot)\bigcdot W\big)_Te^{Z(\xi)}}\\
    &=\Expec{\E\big(f(\xi,\cdot)\bigcdot W\big)_T e^{Z(\xi)}\mathfrak{p}\Big(\Big(\big(\partial_\xi^{\balpha^\prime}f(\xi,\cdot)\bigcdot W\big)_T,\partial_\xi^{\balpha^\prime}Z(\xi)\Big)_{\balpha^\prime\preccurlyeq\balpha}\Big)}.
  \end{align*}
\end{lemma}

\begin{proof}
  Fix $\xi_0\in\R^m$. Notice that $\xi\mapsto \big(f(\xi,\cdot)\bigcdot W\big)_T$ is differentiable at $\xi_0$ in probability with derivative $\big(\partial_\xi^{\balpha}f(\xi_0,\cdot)\bigcdot W\big)_T$. This is an easy consequence of standard stability results for stochastic integrals, see \emph{e.g.} \cite[(2.12) Theorem]{Revuz1999}. Hence, $\xi\mapsto\E\big(f(\xi,\cdot)\bigcdot W\big)_Te^{Z(\xi)}$ is differentiable in probability. Since $\E\big(f(\xi,\cdot)\bigcdot W\big)_T e^{Z(\xi)}\mathfrak{p}\Big(\big(\big(\partial_\xi^{\balpha^\prime}f(\xi,\cdot)\bigcdot W\big)_T,\partial_\xi^{\balpha^\prime}Z(\xi)\big)_{\balpha^\prime\preccurlyeq\balpha}\Big)$ is uniformly integrable over a neighborhood of $\xi$ by H\"older's inequality, \eqref{eq:ui_doleans}, and the Burkholder-Davis-Gundy inequality, the expectation of the differential quotient converges to the expectation of the derivative.
\end{proof}

\begin{proposition}\label{prop:smooth}
  Assume that $b$ satisfies \ref{cond:sk}. Then there is a time $T_0>0$ such that, for each $\ell\in\Cloc^{H-}(\R_+,\R^n)$, the map $y\mapsto p_T(\ell;y)$ is $\C^k$ on $\mathbb{R}^n$. Furthermore, there are $\eta>0$ and $\gamma<H$ such that, for any $\balpha\in\N_0^n$ with $|\balpha|\leq k$, we can find constants $c,C>0$ independent of $\ell$ such that
  \begin{equation*}
    \Big|\partial^{\balpha}_y p_T(\ell;y)\Big|\leq e^{C(1+T^\eta\vertiii{\ell}_{\C^\gamma}^2)}\exp\Big(-c\big|y-\ell(T)\big|^2\Big)\qquad\forall\,y\in\R^n
  \end{equation*}
  for all $T\in(0,T_0]$.
\end{proposition}

\begin{proof} 
  By the product representation of \cref{prop:density_girsanov}, it is enough to prove that $y\mapsto \Psi_T(\ell;y)$ is $\C^k$ on $\mathbb{R}^n$. To this end, we make use of the representation in \cref{prop:repres_psi}
  \begin{equation*}
    \Psi_T(\ell;y)=\E\left(\fL_t^{\ell,\x}\bigcdot W\right)_T\exp\left(\int_0^T\Braket{\fL_t^{\ell,\x},K_t^\x}\,dt\right),
  \end{equation*}
  where $\x=\sigma^{-1}\big(y-\ell(T)\big)$ and $K^\x$ was defined in \eqref{eq:definition_k}.
  Let us first verify that this expression falls in the regime of \cref{lem:parameter_doleans_dade} with
  \begin{equation*}
    f(\x,t)=\fL_t^{\ell, \x},\qquad Z(\x) = \int_0^T \Braket{\fL_t^{\ell,\x}, K_t^\x}\,dt
  \end{equation*}
  for sufficiently small $T>0$. In \cref{lem:exponential_bound_l} we have verified that, for each $\rho>0$ and each $K\subset\R^n$ compact, 
  \begin{equation*}
    \sup_{\x\in K}\Expec{\exp\left(\rho\int_0^T \big|f (\x,t)\big|^2\,dt\right)}<\infty,
  \end{equation*}
  provided we choose $T>0$ sufficiently small. By \eqref{eq:wiener_conditioned} we see that
  \begin{equation*}
    D_\x X_t^\x = \frac{2H}{\varrho_H T^{2H}}\int_0^t (T-s)^{H-\frac12}\,ds,\qquad t\in[0,T].
  \end{equation*}
  Consequently, we obtain
  \begin{equation*}
    D_\x\fL_t^{\ell,\x}=\frac{2H}{\varrho_H T^{2H}}\sigma^{-1}\cI_+^{\frac12-H}\Big[\mathscr{A}_H Db\Big(\ell+\sigma\cI_+^{H-\frac12} X^\x\Big)\Big](t)\sigma, \qquad t\in[0,T],
  \end{equation*}
  where 
  \begin{equation*}
    \mathscr{A}_H(t)\define\cI_+^{H-\frac12}\left(\int_0^\cdot (T-s)^{H-\frac12}\,ds\right)(t)=\frac{1}{\Gamma\left(H+\frac12\right)}\int_0^t (t-s)^{H-\frac12}(T-s)^{H-\frac12}\,ds. 
  \end{equation*}
  Higher order derivatives can be computed similarly. Recalling the growth properties imposed by \eqref{eq:growth_sk_1} and \eqref{eq:growth_sk_2} as well as \cref{lem:holder_fractional_derivative,lem:conditioning_holder}, it is by now routine to verify that $\big\|\partial_\x^{\balpha}f(\x,\cdot)\bigcdot W\big\|_{L^p}<\infty$ locally uniform in $\x$ for each $p\geq 1$ and each $|\balpha|\leq k$. 

  In \eqref{eq:upperbound_cauchy} we have verified that $\Expec{\exp\big(\rho Z(\x)\big)}<\infty$ for any $\rho>0$ locally uniform in $\x$, provided we choose $T>0$ sufficiently small. Since the derivative of $K^\x$ in the parameter is given by 
  \begin{equation*}
    D_\x K_t^\x=\frac{2H}{\varrho_H\left(H+\frac12\right) T^{2H}}\left(T^{H+\frac12}-(T-t)^{H+\frac12}\right)\id \qquad t\in[0,T],
  \end{equation*}
  and $\partial_\x^{\balpha} K_t^\x=0$ for any $|\balpha|\geq 2$, one immediately sees that $\big\|\partial_\x^{\balpha} Z(\x)\big\|_{L^p}<\infty$ for each $p\geq 1$, locally uniform in $\x$.

  Hence, \cref{lem:parameter_doleans_dade} applies and we can `na\"ively' differentiate the right-hand side of \eqref{eq:Psicontinue}. The upper bounds on the derivatives (for sufficiently small times $T>0$) can then be established after applying H\"older's inequality: 
  \begin{equation*}
     \big|\partial_y^{\balpha}\Psi_T(\ell;y)\big|\leq\big\|\Psi_T(\x,\ell)\big\|_{L^p}\left\|\mathfrak{p}\Big(\Big(\big(\partial_\x^{\balpha^\prime}f(\x,\cdot)\bigcdot W\big)_T,\partial_\x^{\balpha^\prime}Z(\x)\Big)_{\balpha^\prime\preccurlyeq\balpha}\Big)\right\|_{L^q}
  \end{equation*} 
  for any $p>1$ and $q=\frac{p}{p-1}$. The first factor can be estimated along the same lines as \cref{prop:lower_bound_conditioned_process} \ref{it:conditioned_density_2} and the second factor grows at most exponentially (but not square-exponentially!) in $|\x|$, $\vertiii{\ell}_{\C^\gamma}$, and $\|X^\x\|_{\C^{\beta}}$ ($\beta<\frac12$). This completes the proof.
\end{proof}

\subsection{The Parameter-Dependent Conditional Density}

In this section we generalize \cref{prop:smooth} to the parameter-dependent setup of \cref{thm:smooth}. First we notice that, as a consequence of \cref{lem:innovation_sde_well_posed}, the equation
\begin{equation}\label{eq:phi_lambda}
  \Phi_t^\lambda(\ell)=\ell(t)+\int_0^t b\big(\lambda,\Phi_s^\lambda(\ell)\big)\,ds+\sigma\tilde{B}_t
\end{equation}
has a unique (in law) weak solution for each $\ell\in\Cloc^{H-}(\R_+,\R^n)$ and each $\lambda\in\R^d$. 

\begin{proposition}\label{prop:equiv410parameter}
  Let $H\in(0,1)$, $k\geq 0$, $\sigma\in\Lin{n}$ be invertible, and $b:\R^d\times\R^n\to\R^n$ satisfy \ref{cond:psk_loc}. Then, for each $\ell\in\Cloc^{H-}(\R_+,\R^n)$, each $\lambda\in\R^d$, and each time $T>0$, the solution $\Phi_T^\lambda$ to \eqref{eq:phi_lambda} has a density $p_T^\lambda(\ell;\cdot)$ with respect to the Lebesgue measure. In addition, there are $\gamma<H$ and $\eta>0$ such that the following hold: 
  \begin{itemize}
    \item For each $K\subset\R^d$ and each $(\balpha,\bbeta)\in\N_0^d\times\N_0^n$ with $|\balpha|+|\bbeta|\leq k$, there are a time $T_0>0$ and $C_K,c_K>0$, such that, for each $\ell\in\Cloc^{H-}(\R_+,\R^n)$, 
  \begin{equation*}
    \sup_{\lambda\in K}\Big|\partial^{\alpha}_\lambda\partial^{\bbeta}_y p_T^\lambda(\ell;y)\Big|\leq e^{C_K(1+T^\eta\vertiii{\ell}_{\C^\gamma}^2)}\exp\Big(-c_K\big|y-\ell(T)\big|^2\Big) \qquad\forall\,y\in\R^n
  \end{equation*}
  for all $T\in(0,T_0]$.

  \item If $b$ satisfies even \ref{cond:psk}, then $c_K, C_K, T_0>0$ can be chosen uniformly over $\lambda\in\R^d$ in the above bound.
  \end{itemize}
\end{proposition}

\begin{proof}
  \Cref{prop:density_girsanov,prop:repres_psi} apply for each $\lambda\in\R^d$ and furnish the representation
  \begin{equation*}
    p_T^\lambda(\ell;y)=\frac{1}{\big(\sqrt{2\pi}\varrho_H T^H\big)^n\det(\sigma)}\exp\left(-\frac{\big|\sigma^{-1}\big(y-\ell(T)\big)\big|^2}{\varrho_H^2 T^{2H}}\right)\Psi^\lambda_T(\ell;y),\qquad y\in\R^n,
  \end{equation*}
  where 
  \begin{equation*}
    \Psi_T^\lambda(\ell;y)\define\Expec{\exp\left(\int_0^T\big\langle\fL^{\lambda,\ell,\x}_t,dX_t^\x\big\rangle-\frac{1}{2}\int_0^T\big|\fL^{\lambda,\ell,\x}_t\big|^2\,dt\right)},\qquad y\in\R^n,
  \end{equation*}
  with the semimartingale of \cref{cor:condition_wiener} and
  \begin{equation*}
    \fL^{\lambda,\ell,\x}_t\define(\varrho_H\sigma)^{-1}\left(\cI_+^{\frac12-H} b\Big(\lambda,\ell+\varrho_H\sigma\cI_+^{H-\frac12}X^\x\Big)\right)(t),\qquad t\in[0,T].
  \end{equation*}
  Similar to \cref{prop:smooth}, one can check that \cref{lem:parameter_doleans_dade} applies and $\Psi_T^\lambda(\ell;y)$ can be na\"ively differentiated. The Gaussian upper bound on $\big|\partial^{\alpha}_\lambda\partial^{\bbeta}_y p_T^\lambda(\ell;y)\big|$ follows as in \cref{prop:smooth}.
\end{proof}

\section{Proofs of the Main Results}\label{sec:proofmain}

We are now in the position to prove the main results of the article. The arguments are all based on \cref{prop:integral_density,prop:lower_bound_conditioned_process,prop:smooth} and suitable bounds on the generalized initial condition $\mu$ with respect to which we integrate the conditional density. The details are delegated to the subsequent subsections.

\subsection{Proofs of \cref{thm:density,thm:smooth}}\label{sec:smoothness}

\begin{proof}[Proof of \cref{thm:density}]
\Cref{prop:tv_convergence_invariant_measure} ensures the existence of  a unique stationary path space law $\prob_\pi$ to the equation \eqref{eq:sde}. 
The existence of a density $p_\infty$ is a direct consequence of \cref{prop:tv_convergence_invariant_measure,prop:density_girsanov} applied with $\ell=z+\sigma\cP^H w$ (the fact that $\ell$ belongs to $\C^{H-}(\R^n,\R^n)$ for $\sW(dw)$-a.e. $w\in\cB_H$ follows from \cite[Lemma 6.5(i)]{Deya2019}).

Let us now focus on the Gaussian bounds on $p_\infty$, see \eqref{eq:gaussian_bounds}. Since $b$ is assumed to be Lipschitz continuous, it satisfies \ref{cond:h}. Hence, it follows that the lower and upper bounds of \cref{prop:lower_bound_conditioned_process} hold true for a given $T_0>0$. Thus, denoting slightly abusively $p_{T_0}(z,w;\cdot)$ the density of  \eqref{eq:innovation_sde} at time $T_0$ when $\ell=z+\cP^Hw$, there exist some positive constants $c_{T_0}$ and $C_{T_0}$ such that, for any $y\in\R^n$,
\begin{align}
p_{T_0}(z,w;y)&\ge   {C_{T_0} e^{-2c_{T_0}\vertiii{z+\sigma\cP^H w}^2_{\C^\gamma}}}\exp\left(-{c_{T_0}\big|y-z-\sigma\cP^H w(T_0)\big|^2}\right) \\
&\ge C_{T_0} \exp\left(-\tilde{c}_{T_0}\left(|y|^2+ |z|^2+\|\cP^Hw\|_{\infty}^2+\|\cP^H w\|_{\C^\gamma}^2\right)\right).\label{eq:lowerbb}
\end{align}
where $\tilde{c}_{T_0}$ denotes another positive constant depending only on $T_0$ and the norms are taken over the interval $[0,T_0]$. Notice that if $Z$ is a non-negative random variable, which is finite on a set of positive probability, then $\expec{e^{-Z}}>0$. Therefore, the lower bound follows from \cref{prop:integral_density} since $\|\cP^H w\|_{\infty}^2+\|\cP^H w\|_{\C^\gamma}^2$ is $\sW(dw)$-a.s. finite for any $\gamma\in(0,H)$. 

As a next step, we prove the Gaussian-type upper bound. Again by \cref{prop:lower_bound_conditioned_process} and elementary inequalities,  there exist $T_0>0$ and constants $c, C>0$ such that, for any $T\in(0,T_0]$ and any $\varepsilon\in(0,1)$, 
\begin{equation*}
p_T(z,w;y)\leq\frac{Ce^{-c\varepsilon\frac{|y|^2}{T^{2H}}} }{T^{nH}}\exp\left(CT^\eta \big(|z|^2+\|\cP^H w\|_{\infty}^2+\|\cP^H w\|_{\C^\gamma}^2\big)+\frac{c\varepsilon \big|z+\cP^H w(T)\big|^2}{(1-\varepsilon) T^{2H}}\right).
\end{equation*}
Notice that, for any $\gamma<H$, $\|\cP^H w\|_{\infty}=\sup_{t\in[0,T]} |\cP^H w(t)|\le \|\cP^H w\|_{\C^\gamma} T^\gamma$
and since $w\mapsto \cP^H w$ is a bounded linear application from $\cB_H$ to $\mathfrak{E}_\gamma^2\big([0,T],\R^n\big)$ (see \eqref{eq:defmathfrakekgam} and the lines below for background), it follows in particular that there is a finite constant $C$ such that $\|\cP^H w\|_{\C^\gamma}\le C \|w\|_{\cB_H}$ on $[0,T_0]$. From these remarks and from the previous inequality, we deduce that for any $\rho>0$, we can find $T_0$ and $\varepsilon$ small enough in such a way that the following bound holds true for some positive constants $c$ and $C$:
\begin{equation} \label{eq:lebesguedom}
p_{T_0}(z,w;y)\le   C\exp\left(\rho (|z|^2+\| w\|_{\cB_H}^2)\right) e^{-c{|y|^2}}. 
\end{equation}
By \cref{cor:invariant_measure}, we thus deduce that $p_\infty(y)\le C e^{-c|y|^2}$. By \cref{prop:smooth}, $y\mapsto p_{T_0}(z,w;y)$ is continuous under \hyperlink{sk}{$\mathbf{(S)}_{0,H}$}. This property transfers to $p_\infty$ by the Lebesgue continuity theorem and the inequality \eqref{eq:lebesguedom}.

Let us now turn to the derivatives.  By the upper bound given in \cref{prop:smooth}, we obtain with the same arguments as before (under \ref{cond:sk}) the existence of a positive $\rho$ such that for any $\balpha\in\N_0^n$ with $|\balpha|\leq k$,
\begin{align} \label{eq:lebesguedomderiv}
\left| \partial_y^{\balpha}  p_{T_0}(z,w;y)\right|\le   C\exp\left(\rho (|z|^2+\| w\|_{\cB_H}^2)\right) e^{-c{|y|^2}}. 
\end{align}
Thus, since \cref{prop:smooth} also ensures that $y\mapsto \partial_y^{\balpha}p_{T_0}(z,w;y)$ is ${\cal C}^k$, the Lebesgue differentiation theorem combined with \cref{cor:invariant_measure} and \eqref{eq:lebesguedomderiv} ensures in turn that $p_\infty$ is ${\cal C}^k$. The upper bound on the density is also a direct consequence of \cref{cor:invariant_measure} and \eqref{eq:lebesguedomderiv}.
\end{proof}

\begin{proof}[Proof of \cref{thm:smooth}] 
  The proof of this theorem has been prepared in \cref{prop:equiv410parameter} and follows from the same arguments as the ones of \cref{thm:density}.
\end{proof}

\subsection{Proof of \cref{thm:lower_bound_tindel}}\label{sec:proof_tindel}

The main additional difficulty in the proof of this theorem comes from the fact that the upper bound established in \cref{prop:smooth} is only available for a small enough $T_0>0$. Roughly speaking, this constraint on $T_0$ comes from the fact that if $Z\sim {\cal N}(0,1)$,  $\Expec{e^{\rho Z^2}}<\infty$ if and only if $\rho<1/2$ (this in turn implies that similar moments of the Liouville process are only finite for $\rho$ small enough). In the classical Markovian setting, such a problem can be overcame with the use of the so-called Chapman-Kolmogorov equation which allows to implement a ``bootstrapping'' argument and in turn to extend the property to any $T>0$. The challenge in our setting is to adapt such a genuinely Markovian approach with the help of our infinitely-dimensional Markovian structure, see \cref{sec:invariant_measures}. 

We begin with several technical lemmas and then turn to the proof of \cref{thm:lower_bound_tindel}.

\begin{lemma}\label{lem:sum_subgaussian}
  Let $X$ and $Y$ be $\R^n$-valued random variables. Suppose that $Y$ is centered Gaussian and there are $c_X,C_X>0$ and $m_X\in\R^n$ such that the density of $X$ satisfies
  \begin{equation*}
     p_X(x)\leq C_X e^{-c_X |x-m_X|^2}\qquad\forall\, x\in\R^n.
  \end{equation*} 
  Then there is a $C>0$ such that, for each $a\in\R^n$,
  \begin{equation*}
    \Expec{e^{-|X-Y-a|^2}}\leq C e^{-\frac{c_X}{C}|a-m_X|^2}.
  \end{equation*}
\end{lemma}
\begin{proof}
  There is no loss of generality in assuming $m_X=0$. Notice that, for each $\varepsilon\in(0,1)$,
  \begin{equation*}
     \Expec{e^{-|X-Y-a|^2}}\leq e^{-\varepsilon |a|^2}\Expec{e^{\frac{4\varepsilon}{1-\varepsilon}|X|^2}}^{\frac12}\Expec{e^{\frac{4\varepsilon}{1-\varepsilon}|Y|^2}}^{\frac12}
  \end{equation*} 
  and the lemma follows at once.
\end{proof}

\begin{lemma}\label{lem:holder_norm_bar_b}
  Let $(\bar{B}^t_h)_{h\geq 0}$ be the history process defined in \eqref{eq:innovation_process} and let $T>0$. Then, for each $\gamma<H$, there is a $\rho>0$ such that
  \begin{equation*}
    \sup_{t\geq 0}\Expec{e^{\rho\vertiii{\bar{B^t}}_{\C^\gamma}^2}}<\infty,
  \end{equation*}
  where the H\"older norm is taken over the interval $[0,T]$.
\end{lemma}
\begin{proof}
  This is an immediate consequence of the fact $(\bar{B}^t_h)_{h\geq 0}\overset{d}{=}(\bar{B}_h)_{h\geq 0}$ for each $t\geq 0$ and \cref{prop:fernique}.
\end{proof}

Henceforth, let $Y^{y_0}$ denote the solution to \eqref{eq:sde} started from the generalized initial condition $\delta_{y_0}\otimes\sW$.

\begin{lemma}\label{lem:gaussian_moments_solution}
  Assume \ref{cond:h} and let $K\subset\R^n$ be compact. Then, for each $T>0$ there is a $\rho>0$ such that
  \begin{equation*}
    \sup_{y_0\in K}\sup_{t\in[0,T]}\Expec{e^{\rho |Y_t^{y_0}|^2}}<\infty.
  \end{equation*}
\end{lemma}
\begin{proof}
  A straight-forward Gr\"onwall argument shows that 
  \begin{equation*}
    \big|Y_t^{y_0}\big|\leq e^{CT}\big(1+ |y_0|+\|B\|_\infty\big)\qquad\forall\,t\in[0,T]
  \end{equation*}
  and the lemma follows from \cref{prop:fernique}.
\end{proof}

Finally, we can prove the last main result of this article: 

\begin{proof}[Proof of \cref{thm:lower_bound_tindel}]
By \cref{prop:strong_well_posedness_tindel}, strong existence and uniqueness hold under \ref{cond:h}. Let $\Phi(y_0,w)$ denote the unique (in law) weak solution to \eqref{eq:innovation_sde} with $\ell=y_0+\cP^H w$, see \cref{lem:innovation_sde_well_posed}. Then, by \cref{prop:density_girsanov}, we know that, for any $t>0$, the density of   $\Phi_t(y_0,w)$ exists for any $(y_0,w)\in\R^n\times \cB_H$. We denote it by $p_t(y_0,w;\cdot)$. Thus, by \cref{prop:integral_density} applied with the generalized initial condition $\mu=\delta_{y_0}\otimes \sW$, $Y^{y_0}_t$ also admits a density given by 
$$\bar{p}_t(y_0;y)=\int_{\cB_H} p_t(y_0,w;y)\sW(dw)\quad\forall y\in\mathbb{R}^n.$$
Let us now turn to the lower bound in \eqref{eq:lower_bound_tindel}. Employing \cref{prop:lower_bound_conditioned_process} \ref{it:conditioned_density_1} and similar arguments as in \eqref{eq:lowerbb}, we have
  \begin{equation*}
    \bar{p}_t(y_0;y)\geq \frac{C e^{-2c|y_0|^2}}{t^{nH}}\exp\left(-\frac{2c\big|y-y_0\big|^2}{t^{2H}}\right)\Expec{\exp\left(-2c|\sigma|\left(\vertiii{\bar{B}}_{\C^\gamma([0,t])}^2+\frac{|\bar{B}_t|^2}{t^{2H}}\right)\right)}
  \end{equation*}
  locally uniform in $t$. Since
  \begin{equation*}
    \Expec{\exp\left(-2c|\sigma|\left(\vertiii{\bar{B}}_{\C^\gamma([0,t])}^2+\frac{|\bar{B}_t|^2}{t^{2H}}\right)\right)}=\Expec{\exp\left(-2c|\sigma|\left(t^{2H}\vertiii{\bar{B}}_{\C^\gamma([0,1])}^2+|\bar{B}_1|^2\right)\right)}
  \end{equation*}
  by the fractional Brownian scaling and the right-hand side is manifestly uniformly bounded away from zero over compact time intervals, the first part of the proof is complete. 

  \medskip

  Next, we turn to the upper bound in \eqref{eq:lower_bound_tindel}.  As indicated in the preamble of this section, the argument proceeds by `bootstrapping' the small-time estimate provided in \cref{prop:lower_bound_conditioned_process} \ref{it:conditioned_density_2} with the help of the fractional Chapman-Kolmogorov equation, see \cref{lem:fraction_chapman_kolmogorov}. Let $K\subset\R^n$ be compact and choose $y_0\in K$. Let $\rho_1>0$ be the constant of \cref{lem:gaussian_moments_solution}. Let $C,T_0>0$ be the constants furnished by \cref{prop:lower_bound_conditioned_process} \ref{it:conditioned_density_2} and let $\rho_2>0$ be the constant provided by \cref{lem:holder_norm_bar_b} applied with the interval $[0,T_0]$. Define $\ft\define\left(\frac{\rho_1\wedge\rho_2}{6C}\right)^{\frac1\eta}\wedge T_0$ and set $I_k\define\bigcup_{i=1}^k \big((i-1)\ft, i \ft\big]$. We show that, for each $k\in\N$, there are $c_k,C_k>0$ such that 
  \begin{equation}\label{eq:bar_p_induction}
    \bar{p}_t(y_0;y)\leq C_k e^{-c_k|y-y_0|^2}\qquad\forall\,y\in\R^n,\,t\in I_k.
  \end{equation}
  To this end, we induct on $k\in\N$. For the induction base we notice that, if $t\leq\ft\leq T_0$, we obtain from \cref{prop:lower_bound_conditioned_process} \ref{it:conditioned_density_2} and the Cauchy-Schwarz inequality
  \begin{align*}
    \bar{p}_t(y_0;y)&\leq \frac{e^{C(1+2t^\eta|y_0|^2)}}{t^{nH}}\Expec{\exp\left(-\frac{c\big|y-y_0-\sigma\bar{B}_t\big|^2}{t^{2H}}+2Ct^\eta\vertiii{\bar{B}}_{\C^\gamma}^2\right)} \\
    &\leq \frac{e^{C(1+2t^\eta|y_0|^2)}}{t^{nH}}\Expec{\exp\left(-\frac{2c\big|y-y_0-\sigma\bar{B}_t\big|^2}{t^{2H}}\right)}^{\frac12}\Expec{\exp\left(4Ct^\eta\vertiii{\bar{B}}_{\C^\gamma}^2\right)}^{\frac12}.
  \end{align*}
  Since the second expectation is finite by virtue of our choice of $\ft\leq\left(\frac{\rho_2}{6C}\right)^{\frac1\eta}$, it is elementary to verify \eqref{eq:bar_p_induction} for $k=1$. 

  Let now $t\in\big(k \ft,(k+1)\ft\big]$ and set $\Delta=t-k\ft$. Then 
  \begin{align*}
      \bar{p}_t(y_0;y)&=\Expec{p_{\Delta}\big(Y_{kT_0}^{y_0}+\sigma\bar{B}^{kT_0};y\big)} \\
      &\leq \Expec{\exp\left(C\left(1+\Delta^\eta\vertiii{Y_{kT_0}^{y_0}+\bar{B}^{kT_0}}_{\C^\gamma}^2\right)\right)\exp\left(-c\big|y-Y_{kT_0}^{y_0}-\bar{B}^{kT_0}_\Delta\big|^2\right)} \\
      &\leq e^C\Expec{\exp\left(6C\Delta^\eta\big|Y_{kT_0}^{y_0}\big|^2\right)}^{\frac13}\Expec{\exp\left(6C\Delta^\eta\vertiii{\bar{B}^{kT_0}}_{\C^\gamma}^2\right)}^{\frac13} \\
      &\phantom{\leq}\times\Expec{\exp\left(-3c\big|y-Y_{kT_0}^{y_0}-\bar{B}^{kT_0}_\Delta\big|^2\right)}^{\frac13},
  \end{align*}
  where we used H\"older's inequality. Since $6C\Delta^\eta\leq \rho_1\wedge \rho_2$ by construction, the first two expectations are finite by \cref{lem:gaussian_moments_solution,lem:holder_norm_bar_b}, respectively. Finally, there is a (potentially decreased, but locally uniform in time!) $c>0$ such that
  \begin{equation*}
    \Expec{\exp\left(-3c\big|y-Y_{kT_0}^{y_0}-\bar{B}^{kT_0}_\Delta\big|^2\right)}^{\frac13}\lesssim \Expec{\exp\left(-c\big|y-y_0\big|^2\right)}.
  \end{equation*}
  This follows from the induction hypothesis and \cref{lem:sum_subgaussian}. The proof is complete.  
\end{proof}

\AtNextBibliography{\small}
\printbibliography


\end{document}